\documentclass{jloganal}
\usepackage[british]{babel}
\usepackage{amssymb,stmaryrd,enumerate,hyperref,cleveref,xcolor}
\usepackage[cal=boondoxo]{mathalfa}
\usepackage[pagewise]{lineno}

\hypersetup{	
    colorlinks,
    citecolor=violet,
    filecolor=black,
    linkcolor=blue,
    urlcolor=teal
}

\title{On ordered groups of regular growth rates}

\author{Vincent Mamoutou Bagayoko}
\surname{Bagayoko}
\address{IMJ-PRG\\
Universit{\'e} Paris Cit{\'e}\\
B{\^a}timent Sophie Germain\\
France}
\email{bagayoko@imj-prg.fr}
\urladdr{https://vincentbagayoko.neocities.org}


\volumenumber{}
\issuenumber{}
\publicationyear{}
\papernumber{}
\startpage{}
\endpage{}
\doi{}
\MR{}
\Zbl{}
\received{}
\revised{}
\accepted{}
\published{}
\publishedonline{}
\proposed{}
\seconded{}
\corresponding{}
\editor{}
\version{}

\newcommand{\assign}{:=}

\newcommand{\nasymp}{\not\asymp}
\newcommand{\nin}{\not\in}
\newcommand{\nobracket}{}
\newcommand{\of}{:}

\newcommand{\suchthat}{:}
\newcommand{\tmem}[1]{{\em #1\/}}
\newcommand{\tmop}[1]{\ensuremath{\operatorname{#1}}}
\newcommand{\tmstrong}[1]{\textbf{#1}}
\newcommand{\tmtextit}[1]{\text{{\itshape{#1}}}}
\newenvironment{descriptioncompact}{\begin{description} }{\end{description}}
\newenvironment{enumeratealpha}{\begin{enumerate}[a{\textup{)}}]}{\end{enumerate}}
\newenvironment{itemizedot}{\begin{itemize} }{\end{itemize}}
\newcounter{nnacknowledgments}

{\theoremstyle{remark}\newtheorem{acknowledgments*}[nnacknowledgments]{Acknowledgments}}
{\theoremstyle{remark}\newtheorem{question}{Question}}
\newtheorem{theorem}{Theorem}[section]
\newtheorem{maintheorem}{Theorem}
\newtheorem{corollary}[theorem]{Corollary}
\newtheorem{definition}[theorem]{Definition}
\newtheorem{proposition}[theorem]{Proposition}
{\theoremstyle{remark}\newtheorem{example}[theorem]{Example}}
{\theoremstyle{remark}\newtheorem{remark}{Remark}}
\newtheorem{lemma}[theorem]{Lemma}
\newcommand{\nonconverted}[1]{\mbox{}}

\begin{document}

\begin{abstract}
  We introduce an elementary class of linearly ordered groups, called growth
  order groups, encompassing certain groups under composition of formal series
  (e.g. transseries) as well as certain groups $\mathcal{G}_{\mathcal{M}}$ of
  infinitely large germs at infinity of unary functions definable in an
  o-minimal structure $\mathcal{M}$. We study the algebraic structure of
  growth order groups and give methods for constructing examples. We show that
  if $\mathcal{M}$ expands the real ordered field and germs in
  $\mathcal{G}_{\mathcal{M}}$ are levelled in the sense of Marker \& Miller,
  then $\mathcal{G}_{\mathcal{M}}$ is a growth order group.
\end{abstract}

\maketitle

\section*{Introduction}

How do two quantities that grow regularly toward infinity behave under
composition? How to characterise the order of growth of such magnitudes?

Hardy introduced {\cite{H1910}} L-functions, which are real-valued functions
obtained as combinations of the exponential function, the logarithm and
semialgebraic functions. They naturally form a differential ring under
pointwise operations. More remarkably, Hardy showed that any two such
functions can always be compared on small enough neighborhoods of $+ \infty$. That is, germs at $+ \infty$ of L-functions are linearly ordered.
For instance, the inequalities
\[ \exp (t) > t^n > \cdots > t^2 > nt > \cdots > 2 t > t + n > \cdots > t + 1
   > t \]
hold on positive half lines $(a, + \infty) \subseteq \mathbb{R}$. Differential-algebraic equations and inequalities, and indeed the whole
first-order theory of fields of germs in the language of ordered valued
differential fields, are well understood {\cite{vdH:mt,ADH:H-closed}}.

The compositional theory of such quantities, however, is unknown. If $f, g$
are two real-valued functions and $g$ eventually exceeds all constant
functions, then the germ of $f \circ g$ only depends on that of $f$ and that
of $g$. This induces a law of composition of germs. Even short and simple
functional equations, involving germs of even regular commonplace
functions... turn out to be particularly difficult to analyse.
In particular, when is the simple inequality
\begin{equation}
  f \circ g > g \circ f \label{eq-basicbasic}
\end{equation}
satisfied for two germs $f, g$ of $L$-functions? We will define a first-order theory of ordered groups of abstract
regular growth rates, that describes in
particular the solutions of \eqref{eq-basicbasic} in said groups.

Let us see how the informal notion of regular
growth rate can be instantiated. The most concrete example is that of elements in Hardy fields {\cite{Bou07}}, i.e. ordered
differential fields of germs. If a Hardy field $\mathcal{H}$ is closed under
compositions, and if its subset $\mathcal{H}^{>\mathbb{R}}$ of germs that lie
above all constant germs is closed under functional inversion, then
$\mathcal{H}^{>\mathbb{R}}$ is an ordered group.

Given an o-minimal structure $\mathcal{M}$, the set
$\mathcal{M}_{\infty}$ of germs at $+ \infty$ of unary definable functions in
$\mathcal{M}$ is also linearly ordered by eventual comparison. Its subset
$\mathcal{G}_{\mathcal{M}}$ of germs of functions that tend to $+ \infty$ at
$+ \infty$ is an ordered group for the induced ordering and the composition of
germs, and the asymptotic growth of germs in $\mathcal{G}_{\mathcal{M}}$ is strongly related \cite{MilStar:growth} to the algebra of definable sets in $\mathcal{M}$. Whereas $\mathcal{M}_{\infty}$ is model theoretically tame provided
$\mathcal{M}$ has definable Skolem functions (see \Cref{subsection-o-minimal}), the ordered group
$\mathcal{G}_{\mathcal{M}}$ is not interpretable in $\mathcal{M}$ in general,
and its first-order theory in the language
$\mathcal{L}_{\tmop{og}}$ of ordered groups is not tame in general. Thirdly, consider an ordered field $\mathbb{K}$ of generalised power series
{\cite{Hahn1907}} over an ordered field of constants $C$, whose set
$\mathbb{K}^{> C}$ of series lying above all constants is non-empty. In
certain cases, there is a composition law $\mathord{\circ} \of \mathbb{K}
\times \mathbb{K}^{> C} \longrightarrow \mathbb{K}$ such that $(\mathbb{K}^{>
C}, \circ, <)$ is an ordered group. Examples include fields of transseries {\cite{vdH:ln,vdDMM01}, fields of hyperseries
{\cite{Bag:hyperclosed}}, and, conjecturally {\cite[Conclusion, 1]{Bag:phd}},
Conway's field of surreal numbers {\cite{Con76}}. Groups of the form
$\mathcal{H}^{>\mathbb{R}}$, $\mathcal{G}_{\mathcal{M}}$ and $\mathbb{K}^{>
C}$ share important first-order properties in $\mathcal{L}_{\tmop{og}}$. No
systematic study of this resemblance has been done yet, and this paper can be taken as a primer on that matter.

We propose a first-order theory
$T_{\tmop{gog}}$ in $\mathcal{L}_{\tmop{og}}$ whose models are ordered groups of abstract regular growth rates. We call them {\tmem{growth
order groups}}. Simple examples include Abelian ordered groups, and, for
instance, ordered groups of strictly increasing affine maps on an ordered
vector space. We will show that models of
$T_{\tmop{gog}}$ comprise both groups of o-minimal germs, groups of formal
series and more abstract examples, and that $T_{\tmop{gog}}$ is sufficiently
strong to provide insight on these groups that is not readily deducible from
their concrete presentations.

In the first section, we give our conventions and notations for ordered
groups, that are always linearly left-ordered and right-ordered. We state
well-known basic facts about such groups, taking from
{\cite{Levi:ord-grp,Neu:or-gr,Fuchs11}}.

In \Cref{section-gogs}, we introduce the three axioms \hyperref[gog1]{\textbf{GOG1}}--\hyperref[gog3]{\textbf{GOG3}}
for growth order groups, starting with \hyperref[gog1]{\textbf{GOG1}} and \hyperref[gog2]{\textbf{GOG2}}
(\Cref{subsection-growth-axioms}).
\Cref{subsection-valuation-theory} focuses
on the existence of a non-commutative valuation, in the sense of {\cite{SaSo:vgroups}}, on ordered groups
satisfying \hyperref[gog1]{\textbf{GOG1}}. We then define {\tmem{scaling elements}}
(\Cref{subsection-scaling-elements}), which form scales along which elements
in the group have asymptotic expansions as in classical valuation theory. In
\Cref{subsection-growth-order-groups}, we
introduce the final axiom \hyperref[gog3]{\textbf{GOG3}} and we show that
growth order groups are commutative transitive {\cite{FiRos:CT}}, that is:

\begin{maintheorem}
  {\tmem{[\Cref{cor-CT}]}} The centraliser of a non-trivial element in a
  growth order group is Abelian.
\end{maintheorem}

We also discuss the existence of asymptotic expansions in growth order
groups, and embeddings of growth order groups into groups of non-commutative
formal series (\Cref{subsection-structure}).

\Cref{section-constructions} gives methods for constructing growth order
groups. We give conditions under which the quotient of a
growth order group is a growth order group (\Cref{subsection-quotients}). We then define the ordered groups
$\mathcal{G}_{\mathcal{M}}$ of germs in an o-minimal structure $\mathcal{M}$
and give examples where $\mathcal{G}_{\mathcal{M}}$ is, or is not a growth
order group (\Cref{subsection-o-minimal}).

In \Cref{section-H-fields}, we give conditions on an
o-minimal expansions $\mathcal{R}$ of the real ordered field for
$\mathcal{G}_{\mathcal{R}}$ to be a growth order group. Let $\mathcal{R}$ be an o-minimal expansion of the real ordered field.
Given a real-valued germ $g$ at $+ \infty$ and $n \in \mathbb{N}$, we write
$g^{[n]}$ for the $n$-fold compositional iterate of $g$. With
{\cite{Rosli83:rank,MarMil97}}, we say that $\mathcal{R}$ is
{\tmem{levelled}}{\index{levelled o-minimal field}} if for all positive
elements $f$ of the ordered group $\mathcal{G}_{\mathcal{R}}$, there is an $l
\in \mathbb{N}$ such that for all sufficiently large $k \in \mathbb{N}$, we
have
\[ - 1 \leqslant \log^{[n]} \circ f - \log^{[n - l]} \leqslant 1. \] E.g. $l=0$ for the germ of the identity or the function $0<t\mapsto \exp(\log(t)^2)$, and $l=1$ for the germ of $\exp$ or $\exp^2$.  The main theorem is as follows:

\begin{maintheorem}
  \label{th-main}Let $\mathcal{R}$ be an o-minimal expansion of the real
  ordered field. If $\mathcal{R}$ is levelled, then
  $\mathcal{G}_{\mathcal{R}}$ is a growth order group. Moreover, centralisers
  of non-trivial elements in $\mathcal{G}_{\mathcal{R}}$ are Archimedean.
\end{maintheorem}

Many o-minimal expansions of $\mathbb{R}$ are levelled,
including expansions of $\mathbb{R}$ by generalised analytic classes and the
exponential {\cite{RoSer:QEqaa,RoSerSpei22}} (see \Cref{cor-levelled-qaa}),
and the Pfaffian closure of the real ordered field {\cite{Spei:Pfaff}} (see
\Cref{cor-levelled-Pfaff}). In fact, no o-minimal expansion of $\mathbb{R}$ is
known not to be levelled.

Our proof heavily relies on the fact that the elementary extension
$\mathcal{R}_{\infty}$ of $\mathcal{R}$ is closed under derivation of germs,
and that as an ordered valued differential field, it is an H-field
{\cite{AvdD02}}. In \Cref{subsection-H-fields-with-composition}, we introduce a first-order
theory of H-fields $K$ over an ordered field of constants $C$ with a
composition law $\mathord{\circ} \of K \times K^{> C} \longrightarrow K$ and a
compositional identity $x \in K^{> C}$, such that $(K^{> C}, \circ, x, <)$ is
an ordered group. A crucial feature of such fields is that they satisfy the axiom
scheme of Taylor expansions (\hyperref[hfc5]{\textbf{HFC5}}). 
We prove in \Cref{subsection-Taylor} that certain Hardy fields closed under composition have Taylor expansions. Say that a real-valued function $f$ is
{\tmem{transexponential}}{\index{transexponential function}} if the germ of
$f$ lies above $\exp^{[n]}$ for each $n \in \mathbb{N}$. We show in particular
that:

\begin{maintheorem}\label{th-3}
  {\tmem{[\Cref{cor-elem-Hfield}]}} Let $\mathcal{R}$ be an o-minimal
  expansion of an ordered field. Assume that $\mathcal{R}$ has an elementary
  substructure $\mathcal{R}_0$ with underlying ordered field $\mathbb{R}$ and
  that $\mathcal{R}_0$ defines no transexponential function. Then
  $\mathcal{R}_{\infty}$ has Taylor expansions.
\end{maintheorem}

Using Taylor expansions, we derive conjugacy relations in H-fields with
composition and inversion (Sections~\ref{subsection-conjugacy} and~\ref{subsection-proofofmain}). In the case when $C =\mathbb{R}$, this allows
us to prove a general result (\Cref{th-levelled}) giving conditions under
which $K^{>\mathbb{R}}$ is a growth order group. \Cref{th-main} follows from applications of \Cref{th-levelled}. We rely on Miller's first dichotomy result {\cite{Mil:dicho}} stating that
either each germ in $\mathcal{R}_{\infty}$ is bounded by the germ of a
polynomial function, or $\mathcal{R}$ defines the exponential function. The
polynomially bounded and exponential cases are treated in
Sections~\ref{subsection-application-bounded}
and~\ref{subsection-application-exp} respectively.

\section{Ordered groups}

\subsection{Ordered groups}

\begin{definition}
  \label{def-bi-ordered}An {\tmem{{\tmstrong{ordered group}}}}{\index{ordered
  group}} is a group $(\mathcal{G}, \cdot, 1)$ together with a
  {\underline{linear}} (i.e. total) ordering $<$ on $\mathcal{G}$ such that
  \begin{equation}
    \forall f, g, h \in \mathcal{G}, (g > h \Longrightarrow (fg > fh \wedge gf
    > hf)). \label{eq-bi-ordered}
  \end{equation}
\end{definition}

We write $\leqslant$ for the large relation corresponding to $<$, i.e. $f
\leqslant g \Longleftrightarrow (f < g \vee f = g)$. Our first-order
language of ordered groups is $\mathcal{L}_{\tmop{og}} \assign \langle \cdot, 1,
\leqslant, \tmop{Inv} \rangle$\label{autolab1} where the unary function symbol
$\tmop{Inv}$ is to be interpreted as the inverse map $g \mapsto g^{- 1}$. We
write $T_{\tmop{og}}$ for the expected $\mathcal{L}_{\tmop{og}}$-theory of
ordered groups. Homomorphisms should be understood in the model theoretic
sense: a {\tmem{homomorphism}}{\index{homomorphism of ordered groups}} of
ordered groups is a nondecreasing group morphism, whereas an
{\tmem{embedding}}{\index{embedding of ordered groups}}, of ordered groups is
a strictly increasing group morphism.

\begin{remark}
  \label{rem-groups-functions}An ordered group $\mathcal{G}$ can be seen as a
  group of automorphisms of a linearly ordered set $(X, <)$ ordered by
  universal pointwise comparison
  \[ \varphi < \phi \Longleftrightarrow (\forall x \in X, (\varphi (x) < \phi
     (x))) . \]
  Indeed, let $\mathcal{G}$ act on $(\mathcal{G}, <)$ by
  translations on the left. This intuition is particularly relevant in the case of growth order groups.
\end{remark}

Given an ordered group $\mathcal{G}$, we write\label{autolab2}
\label{autolab3}
\[ \mathcal{G}^{>} \assign \{ f \in \mathcal{G} \suchthat f > 1 \}
   \text{{\hspace{3em}}and{\hspace{3em}}$\mathcal{G}^{\neq} \assign \{f \in
   \mathcal{G} \suchthat f \neq 1\}$.} \]
An ordered group $(\mathcal{G}, \cdot, 1, <)$ is said
{\tmem{Archimedean}}{\index{Archimedean ordered group}} if for all $f, g \in
\mathcal{G}^{\neq}$, there is an $n \in \mathbb{Z}$ such that $f^n \geqslant
g$. Recall by H{\"o}lder's theorem (see {\cite[Section~IV.1,
Theorem~1]{Fuchs11}}) that $\mathcal{G}$ is Archimedean if and only if it
embeds into $(\mathbb{R}, +, 0, <)$. In particular, Archimedean
ordered groups are Abelian.

If $(H, \cdot, 1)$ is a group and $f, g \in H$, then we write\label{autolab4}
\[ [f, g] \assign f^{- 1} g^{- 1} fg. \]
We recall that the {\tmem{centraliser}}{\index{centraliser of an element}} of
an element $g \in H$ is the subgroup\label{autolab5}
\[ \mathcal{C} (g) \assign \{ h \in H \suchthat [g, h] = 1 \} = \{ h \in H
   \suchthat hg = gh \} . \]
For each $h \in H$, we have
\begin{equation}
  \mathcal{C} (hgh^{- 1}) = h\mathcal{C} (g) h^{- 1}
  \label{eq-centralizer-conj-identity} .
\end{equation}

\subsection{Powers}

Let $(\mathcal{G}, \cdot, 1, <)$ be an ordered group. Let us make a few
comments on powers of elements in $\mathcal{G}$. The axioms for
ordered groups imply that $\mathcal{G}$ is torsion-free, i.e. $f^n = 1
\Longrightarrow f = 1$ for all $f \in \mathcal{G}$ and $n \in \mathbb{Z}
\setminus \{ 0 \}$.

\begin{lemma}
  {\tmem{{\cite[Lemma~1.1]{Neu:or-gr}}}} For all $f, g \in \mathcal{G}$ and $n
  > 0$, we have $[f^n, g] = 1 \Longrightarrow [f, g] = 1$.
\end{lemma}

\begin{corollary}
  \label{cor-frac-commute}{\tmem{{\cite[Corollary~1.2]{Neu:or-gr}}}} Let $f, g
  \in \mathcal{G}$ and $m, n \in \mathbb{N}^{>}$ with \ $f^n g^m = g^m f^n$.
  Then $fg = gf$.
\end{corollary}

\begin{corollary}
  Let $g \in \mathcal{G}$. Let $m, n \in \mathbb{Z} \setminus \{ 0 \}$ and $f
  \in \mathcal{G}$ with $f^m = g^n$. Then $f$ is unique to satisfy $f^m = g^n$, and we have $[f, g] = 1$.
\end{corollary}

\begin{proof}
  That $f$ is unique follows from the fact that $\mathcal{G}$ is torsion-free.
  We have $[f, g] = 1$ by \Cref{cor-frac-commute}.
\end{proof}

\section{Growth order groups}\label{section-gogs}

We now introduce growth order groups by defining a first-order
theory~$T_{\tmop{gog}} \supseteq T_{\tmop{og}}$ thereof.

\subsection{Growth axioms}\label{subsection-growth-axioms}

Let $(\mathcal{G}, \cdot, 1, <)$ be an ordered group. Consider the following
sentences in $\mathcal{L}_{\tmop{og}}$ (after an obvious rewriting).

\begin{descriptioncompact}
  \item[$\textbf{GOG1}$] \label{gog1}Given $f, g \in \mathcal{G}^{>}$ with $f
  \geqslant g$ and $g_0 \in \mathcal{C} (g)$, there is an $f_0 \in \mathcal{C}
  (f)$ with $f_0 \geqslant g_0$.
  
  \item[$\textbf{GOG2}$] \label{gog2}For $f, g \in \mathcal{G}^{>}$, we have
  \begin{equation}
    f >\mathcal{C} (g) \Longrightarrow fg > gf \text{.} \label{eq-growth}
  \end{equation}
\end{descriptioncompact}

Any ordered Abelian group automatically satisfies \hyperref[gog1]{\textbf{GOG1}}, and vacuously
satisfies \hyperref[gog2]{\textbf{GOG2}}. We say that $\mathcal{G}$ has {\tmem{Archimedean
centralisers}}{\index{Archimedean centralisers}} if for each $g \in
\mathcal{G}^{\neq}$, the ordered group $\mathcal{C} (g)$ is Archimedean.

\begin{proposition}
  \label{prop-Archimedean-gog1}If $\mathcal{G}$ has Archimedean
  centralisers, then \hyperref[gog1]{\textbf{GOG1}} holds.
\end{proposition}

\begin{proof}
  Let $f, g \in \mathcal{G}^{>}$ with $f \geqslant g$ and let $g_0 \in
  \mathcal{C} (g)$. We have $g^{- n} \leqslant g_0 \leqslant g^n$ for a
  certain $n \in \mathbb{N}$, so $f^n$ is an element of $\mathcal{C} (f)$ with
  $f^n \geqslant g^n \geqslant g_0$.
\end{proof}

\subsection{Some non-commutative valuation
theory}\label{subsection-valuation-theory}

In Sections~\ref{subsection-valuation-theory}
and~\ref{subsection-scaling-elements}, we fix an ordered group $(\mathcal{G},
\cdot, 1, <)$ satisfying \hyperref[gog1]{\textbf{GOG1}}. For $f,g \in \mathcal{G}$, we write $f\preccurlyeq g$ if $g\neq 1$ and there are $g_0, g_1 \in \mathcal{C} (g)$ such that $g_0 \leqslant f \leqslant g_1$, i.e.
if $f$ lies in the convex hull of
$\mathcal{C} (g)$. We also set $1\preccurlyeq g$ for all $g \in \mathcal{G}$.

\begin{proposition}
  \label{prop-quasi-ordering}The relation $\preccurlyeq$ is a linear
  quasi-ordering on $\mathcal{G}$.
\end{proposition}

\begin{proof}

Throughout the proof, we consider generic elements $f, g,h \in \mathcal{G}$.

We first prove that the relation is total. 
  We have $f \preccurlyeq g \Longleftrightarrow f^{- 1} \preccurlyeq g
  \Longleftrightarrow f \preccurlyeq g^{- 1} \Longleftrightarrow f^{- 1}
  \preccurlyeq g^{- 1}$. Thus we may assume that $f, g >1$. We either have $f \leqslant g$, in which case $f \preccurlyeq
  g$, or $g \leqslant f$, in which case $g \preccurlyeq f$.

  Now suppose that $f \preccurlyeq g$ and $g \preccurlyeq
  h$. We may assume that $f, g, h \neq 1$. So there are $g_0, g_1 \in
  \mathcal{C} (g)$ and $h_0, h_1 \in \mathcal{C} (h)$ with
  $g_0 \leqslant f \leqslant g_1$ and $h_0 \leqslant g
     \leqslant h_1$.
  We may choose $g_0, h_0 \in \mathcal{G}^{<}$ and $g_1, h_1 \in
  \mathcal{G}^{>}$. By \hyperref[gog1]{\textbf{GOG1}}, there are $h_2, h_3 \in \mathcal{C} (h)$
  with $g_1 \leqslant h_3$ and $g_0^{- 1} \leqslant h_2$, whence $g_0
  \geqslant h_2^{- 1}$. We thus have $h_2^{- 1} \leqslant f \leqslant h_3$,
  i.e. $f \preccurlyeq h$. So $\preccurlyeq$ is transitive. It is clearly
  reflexive.
\end{proof}

We have an equivalence relation $f \asymp g \Longleftrightarrow f \preccurlyeq
g \wedge g \preccurlyeq f$\label{autolab7} on $\mathcal{G}$ or
$\mathcal{G}^{\neq}$. Given $f \in \mathcal{G}$, we write $v (f)$ for
the equivalence class of $f$ for $\asymp$, called its {\tmem{valuation}}{\index{valuation of an element}} and we write $v
(\mathcal{G})$ for the quotient set
\[ v (\mathcal{G}) =\mathcal{G}^{\neq} \mathbin{/}
   \mathord{\asymp} = \{ v (f) \suchthat f \in \mathcal{G}^{\neq} \} .
\]
We write $f \prec g$ if $f \preccurlyeq g$ and $g \nasymp f$\label{autolab8}.

\begin{lemma}
  \label{lem-strict}Let $f, g \in \mathcal{G}$ with $g \neq 1$. We have $g
  \prec f$ if and only if $\mathcal{C} (g) < \max (f, f^{- 1})$.
\end{lemma}

\begin{proof}
  If $f \neq 1$, then this is immediate by definition of $\preccurlyeq$. Since $g \not\prec 1$ and $\mathcal{C} (g) \not < 1$, this yields the result.
\end{proof}

\begin{proposition}
  \label{prop-growth-order}For all $g, h \in \mathcal{G}$, we have:
  \begin{enumeratealpha}
    \item \label{def-growth-order-inv}$g^{- 1} \asymp g$.
    
    \item \label{def-growth-order-prod}$gh \preccurlyeq g$ or $gh \preccurlyeq
    h$.
    
    \item \label{def-growth-order-ord}$1 \leqslant g \leqslant h
    \Longrightarrow g \preccurlyeq h$.
  \end{enumeratealpha}
\end{proposition}

\begin{proof}
  The statement \ref{def-growth-order-inv} follows from the fact that
  $g^{- 1} \in \mathcal{C} (g)$.\quad Assume for contradiction that $gh \succ
  g$ and $gh \succ h$. We must have $g, h \neq 1$. By
  Lemma~\ref{lem-strict}, we deduce that $gh >\mathcal{C} (g)$ or that $gh
  <\mathcal{C} (g)$. So $h >\mathcal{C} (g)$ or $h <\mathcal{C} (g)$. But then
  $\max(h,h^{-1})^{-2}  < gh < \max(h,h^{-1})^2$. This contradicts $gh \succ h$. This shows \ref{def-growth-order-prod}. For
  \ref{def-growth-order-ord} we have $h^{- 1} \leqslant g \leqslant h$ where
  $h, h^{- 1} \in \mathcal{C} (h)$.
\end{proof}

 This shows that the function $v\colon \mathcal{G}^{\neq} \longrightarrow v (\mathcal{G})$ is a valuation
   in the sense of
  {\cite[Section~4.4]{Fuchs11}} and of {\cite[Definition~2.1]{SaSo:vgroups}}. We call $v$ the \tmem{standard valuation} on $\mathcal{G}$.

\begin{proposition}
  \label{prop-growth-order-basic}For $g, h \in \mathcal{G}$, we have $ g \prec h \Longrightarrow gh \asymp hg \asymp h$.
\end{proposition}

\begin{proof}
  We have $gh \preccurlyeq h$ by
  Proposition~\ref{prop-growth-order}(\ref{def-growth-order-prod}). Assume for
  contradiction that $gh \prec h$. By
  Proposition~\ref{prop-growth-order}(\ref{def-growth-order-inv}), we have $h
  = g^{- 1}  (gh) \preccurlyeq g^{- 1} \asymp g \prec h$, or $h = g^{- 1} 
  (gh) \preccurlyeq gh \prec h$: a contradiction. Thus $gh \asymp h$. The
  proof of $hg \asymp h$ is symmetric.
\end{proof}

\begin{proposition}
  \label{prop-convex-asymp}For $f \in \mathcal{G}^{\neq}$, the set $v
  (f) \cap \mathcal{G}^{>}$ is convex.
\end{proposition}

\begin{proof}
  Let $g, h \in \mathcal{G}^{>}$ with $g, h \asymp f$ and let $j \in
  \mathcal{G}$ with $g \leqslant j \leqslant h$. We have $g \preccurlyeq j$
  and $j \preccurlyeq h$ by
  \Cref{prop-growth-order}(\ref{def-growth-order-ord}). So $f \preccurlyeq j$
  and $j \preccurlyeq f$ by \Cref{prop-quasi-ordering}, whence $j \asymp f$.
\end{proof}

We can thus define a linear ordering $<$ on $v
(\mathcal{G})$, where for $g, h
\in \mathcal{G}^{\neq}$, we set $v (g) < v (h)$ if and only if
$g \prec h$, i.e. if $v (g) \cap \mathcal{G}^{>} < v (h) \cap
\mathcal{G}^{>}$.

\begin{definition}
  The {\tmem{{\tmstrong{value set}}}}{\index{value set}} of
  $\mathcal{G}$ is the (order type of the) linearly ordered set $(v (\mathcal{G}), <)$.
\end{definition}

One sees that $\mathcal{G}$ has value set $0$ if and only if it is
trivial, and that non-trivial Abelian ordered groups have value set $1$.

\begin{example}
  Let $\mathcal{R}$ denote the real ordered field. It will
  follow from \Cref{th-main} that $\mathcal{G}_{\mathcal{R}}$ satisfies
  \hyperref[gog1]{\textbf{GOG1}} and has Archimedean centralisers. Therefore, the convex hull of
  $\mathcal{C} (g)$ for $g \in \mathcal{G}_{\mathcal{R}}$ is simply the convex
  hull of the set $g^{[\mathbb{Z}]}$ of iterates of $g$ and its inverse.
  Definable functions in $\mathcal{R}$ are semialgebraic. Any non-trivial semialgeraic function $f$ satisfies $\lim \limits_{t\to+\infty} \frac{f(t)}{rt^q} = 1$ for an $(r,q) \in \mathbb{R}^{\times} \times \mathbb{Q}$ (as it has a Puiseux series expansion). Therefore the valuation of the square function is maximal in $c(\mathcal{G})$ and  $v (2 \tmop{id})$ is
  maximal in $v (\mathcal{G}_{\mathcal{R}}) \setminus \{v
  (\tmop{id}^2)\}$. Applying the same idea to $f-\tmop{id}$ for $f \in \mathcal{G}$, we see that 
  \begin{eqnarray*}
    \{q \in \mathbb{Q} \suchthat q < 1\} & \longrightarrow & v
    (\mathcal{G}) \setminus \{v (\tmop{id}^2), v (2
    \tmop{id})\}\\
    q & \longmapsto & v (\tmop{id} + \tmop{id}^q)
  \end{eqnarray*}
  is an isomorphism of ordered sets. In other words, the value set of
  $\mathcal{G}_{\mathcal{R}}$ is the rational interval $((-
  \infty, 1] \cup \{2\}, <)$.
\end{example}

\begin{lemma}
  \label{lem-small-conjugate}For $f, g \in \mathcal{G}$ with $g \prec
  f$, we have $fgf^{- 1} \prec f$.
\end{lemma}

\begin{proof}
  The conjugation by $f$ is an automorphism of $\mathcal{G}$ and $\prec$ is $\varnothing$-definable in the language of ordered groups.
\end{proof}

Given $g, h \in \mathcal{G}^{\neq}$, we write\label{autolab10}
\[ g \sim h \text{{\hspace{3em}}if and only if{\hspace{3em}}$gh^{- 1} \prec
   g.$} \]

\begin{lemma}
  \label{lem-equivalence}For all $g, h \in \mathcal{G}^{\neq}$, the following
  are equivalent:
  \begin{enumeratealpha}
    \item \label{lem-equivalence-a}$g \sim h$
    
    \item \label{lem-equivalence-b}$gh^{- 1} \prec h$
    
    \item \label{lem-equivalence-c}$hg^{- 1} \prec g$
    
    \item \label{lem-equivalence-d}$h \sim g$.
  \end{enumeratealpha}
\end{lemma}

\begin{proof} Suppose that $g \sim h$, i.e. $g h^{-1} \prec g$. We cannot have $h^{-1} \prec g$ by \Cref{prop-growth-order-basic}, so we also have $g h^{-1} \prec h^{-1} \asymp h$ by \Cref{prop-quasi-ordering}. We deduce that \ref{lem-equivalence-a} and \ref{lem-equivalence-b}
  are equivalent. Likewise \ref{lem-equivalence-c} and \ref{lem-equivalence-d} are equivalent. Since $g h^{-1} \asymp h g^{-1}$ by Proposition~\ref{prop-growth-order}(\ref{def-growth-order-inv}), the statements
  \ref{lem-equivalence-a} and \ref{lem-equivalence-c} are equivalent. This concludes the proof.
\end{proof}

\begin{corollary}
  \label{cor-inverses}For $g, h \in \mathcal{G}^{\neq}$, we have $g \sim h
  \Longleftrightarrow g^{- 1} \sim h^{- 1}$. 
\end{corollary}

Note that for $g,h \in \mathcal{G}$ with $g \sim h$, we have $g
\asymp h$. We have $g h^{-1}\prec h$ by \Cref{lem-equivalence}, so $g = (g h^{-1}) h \asymp h$ by \Cref{prop-growth-order-basic}.

\begin{lemma}
  The relation $\sim$ is an equivalence relation on $\mathcal{G}^{\neq}$.
\end{lemma}

\begin{proof}
  For all $g \in \mathcal{G}^{\neq}$, we have have $1 \prec g$ whence $g \sim
  g$. \Cref{lem-equivalence} implies that $\sim$ is symmetric. Let $f, g, h
  \in \mathcal{G}^{\neq}$ with $f \sim g$ and $g \sim h$. So $f \asymp g
  \asymp h$. We have $fh^{- 1} = (fg^{- 1})  (gh^{- 1})$ where $(fg^{- 1}),
  (gh^{- 1}) \prec f$ so $fh^{- 1} \prec f$ by
  \Cref{prop-growth-order}(\ref{def-growth-order-prod}). So $f \sim h$, i.e. $\sim$ is transitive.
\end{proof}

Given a $g \in \mathcal{G}^{\neq}$, we write $\tmop{res} (g)$\label{autolab11}
for the equivalence class of $g$ for $\sim$ in $\mathcal{G}^{\neq}$. We call
$\tmop{res} (g)$ the {\tmem{residue}}{\index{residue of an element}}
of $g$.

\begin{proposition}
  \label{prop-equiv-equiv}For $g \in \mathcal{G}^{\neq}$, the set $\tmop{res}
  (g)$ is convex.
\end{proposition}

\begin{proof}
  Let $f, h \in \mathcal{G}^{\neq}$ with $f \sim g \sim h$ and $j \in
  \mathcal{G}^{\neq}$ with~$f \leqslant j \leqslant h$. In view of
  \Cref{cor-inverses}, we may assume that $g > 1$. Consider an $s_0 \in
  \mathcal{C} (jg^{- 1})$. Suppose that $jg^{- 1} \geqslant 1$. Since $hg^{-
  1} \geqslant jg^{- 1}$, we find by \hyperref[gog1]{\textbf{GOG1}} an $h_0 \in \mathcal{C}
  (hg^{- 1})$ with $h_0 \geqslant s_0$. Now $hg^{- 1} \prec g$ so $h_0 < g$,
  so $s_0 < g$. This shows that $jg^{- 1} \prec g$, whence $j \sim g$ in that
  case. Suppose now that $jg^{- 1} \leqslant 1$. So $1 \leqslant gj^{- 1}
  \leqslant gf^{- 1}$. But $gf^{- 1} \prec g$ so the same arguments for $gj^{-
  1}$ show that $s_0 < g$, whence $j \sim g$. So $\tmop{res} (g)$ is convex.
\end{proof}

We can thus define a linear ordering $\lessdot$ on $\tmop{res}
(\mathcal{G}) \assign \mathcal{G}^{\neq} / \sim$ given by
\[ \tmop{res} (f) \lessdot \tmop{res} (g) \Longleftrightarrow f < g \wedge f
   \nsim g. \]
We set $\tmop{res} (1) = \{1\}$ and $\{1\} \lessdot \tmop{res} (f)$ for all $f
\in \mathcal{G}^{\neq}$. We also write $f \lessdot g$ whenever $\tmop{res} (f)
\lessdot \tmop{res} (g)$. Although we will not rely on this fact, this is also strict ordering on $\mathcal{G}$

\begin{lemma}
  \label{lem-small-commutator-cases}Let $g, h \in \mathcal{G}^{\neq}$ with $g \sim
  h^{- 1}$ or $g \prec h$. Then $[g, h] \prec h$.
\end{lemma}

\begin{proof}
  First suppose that $g \sim h^{- 1}$. \Cref{lem-equivalence} gives
  $g^{-1}h^{-1},g h \prec h$. So $[g,h] \prec h$ by Proposition~\ref{prop-growth-order}(\ref{def-growth-order-prod}). Suppose now that $g \prec h$. So $\delta \assign h^{- 1} gh
  \prec h$. We obtain
  \[ [g, h] = g^{- 1} h^{- 1} gh = g^{- 1} \delta \prec h \]
  by Proposition~\ref{prop-growth-order}(\ref{def-growth-order-prod}).
\end{proof}

\subsection{Scaling elements}\label{subsection-scaling-elements}

Recall that $\mathcal{G}$ is an ordered group satisfying
\hyperref[gog1]{\textbf{GOG1}}.

\begin{definition}
  We say that an element $\mathcal{s} \in \mathcal{G}^>$ is
  {\tmstrong{{\tmem{scaling}}}}{\index{scaling element}} if $\mathcal{C}
  (\mathcal{s})$ is Abelian, and for all $f \in \mathcal{G}$ with $f \asymp
  \mathcal{s}$, there is a $g \in \mathcal{C} (\mathcal{s})^{\neq}$ with $g
  \sim f$.
\end{definition}

Given a scaling $\mathcal{s}$ and $f\asymp g$, the  element $g$ is unique in $\mathcal{C} (\mathcal{s})$. Indeed, for $h
\in \mathcal{C} (\mathcal{s}) \setminus \{ g \}$, writing $j \assign  hg^{-1}$ we have $f (jg)^{- 1} \asymp
h \asymp f$ by \Cref{prop-growth-order}(\ref{def-growth-order-inv},
\ref{def-growth-order-prod}), so we do not have $h=jg \sim f$. Note that each positive element in an Abelian ordered group is scaling.

\begin{definition}
  We say that $\mathcal{G}$ {\tmem{{\tmstrong{has scaling
  elements}}}}{\index{having scaling elements}} if for all $\rho \in v
  (\mathcal{G})$, there is an $\mathcal{s} \in \rho$ which is scaling.
\end{definition}

\begin{proposition}
  \label{prop-reals}Let $\mathcal{s} \in \mathcal{G}^>$ such that
  $(\mathcal{C} (\mathcal{s}), \cdot, 1, <)$ is isomorphic to $(\mathbb{R}, +,
  0, <)$. Then $\mathcal{s}$ is scaling.
\end{proposition}

\begin{proof}
  Let $f \in \mathcal{G}^{\neq}$ with $f \asymp \mathcal{s}$. If $f \in
  \mathcal{C} (\mathcal{s})$, then we are done. Assume that $f \nin
  \mathcal{C} (\mathcal{s})$ and set $h \assign \sup \{ g \in \mathcal{C} (\mathcal{s}) \suchthat g \leqslant
     f \}$.
  For $\varepsilon \in \mathcal{C} (\mathcal{s}) \cap \mathcal{G}^{>}$, we
  have
  \[ \varepsilon^{- 1} h, h \varepsilon^{- 1} < h < \varepsilon h, h
     \varepsilon, \]
because $(\mathcal{G},\cdot,1,<)$ is an ordered group.
  We deduce that
  \begin{equation}
    \mathcal{C} (\mathcal{s}) \cap \mathcal{G}^{<} < h^{- 1} f <\mathcal{C}
    (\mathcal{s}) \cap \mathcal{G}^{>} . \label{eq-ineq-hf}
  \end{equation}
  Assume for contradiction that $h^{- 1} f \succcurlyeq f$. We have $h \in
  \mathcal{C} (\mathcal{s})$ so $h \asymp \mathcal{s} \asymp f$. By \Cref{prop-growth-order}(\ref{def-growth-order-inv},
  \ref{def-growth-order-prod}) we have $h^{- 1} f \asymp f$. Since $h \in
  \mathcal{C} (\mathcal{s})$ and $f \nin \mathcal{C} (\mathcal{s})$, there are
  $f_0, f_1$ which have the same sign, with $f_0 < h^{- 1} f < f_1$. By
  \hyperref[gog1]{\textbf{GOG1}}, there are $g_0, g_1 \in \mathcal{C} (\mathcal{s})$ which have
  the same sign as well, with $g_0 < h^{- 1} f < g_1$. This contradicts
  \eqref{eq-ineq-hf}. We deduce that $h^{- 1} f \prec f$, i.e. $h \sim f$.
\end{proof}

\begin{lemma}
  \label{lem-small-commutator}Suppose that $\mathcal{s} \in \mathcal{G}^>$
  is scaling. Then for all $f, g \in \mathcal{G}^{\neq}$ with $f \asymp g
  \asymp \mathcal{s}$, we have
  $[f, g] \prec f$.
\end{lemma}

\begin{proof}
  If $f \sim g^{- 1}$, then this follows from
  \Cref{lem-small-commutator-cases}. Assume that $f \nsim g^{- 1}$. Let
  $\mathcal{t}, \mathcal{u} \in \mathcal{C} (\mathcal{s})^{\neq}$ with
  $\mathcal{t} \sim f$ and $\mathcal{u} \sim g$. We have $\mathcal{t} \nsim
  \mathcal{u}^{- 1}$ by \Cref{prop-equiv-equiv}, so
  Proposition~\ref{prop-growth-order}(\ref{def-growth-order-prod}) implies
  that~$\mathcal{t}\mathcal{u} \asymp \mathcal{t} \asymp \mathcal{s}$. Set
  \begin{eqnarray*}
    \varepsilon & \assign & \mathcal{t}^{- 1} f \prec \mathcal{s}\\
    \delta & \assign & g\mathcal{u}^{- 1} \prec \mathcal{s}.
  \end{eqnarray*}
  Recall that $\mathcal{C}(\mathcal{s})$ is Abelian, so $[\mathcal{t},\mathcal{u}]=1$. We have
  \begin{eqnarray*}
    [f, g] & =&  f^{- 1} g^{- 1} fg \\ 
    & = & \varepsilon^{- 1} \mathcal{t}^{- 1}
    \mathcal{u}^{- 1} \delta^{- 1} \mathcal{t} \varepsilon \delta \mathcal{u}  \\ 
    & = & \varepsilon^{- 1} [\mathcal{t},
    \mathcal{u}] (\mathcal{u}^{-1} (\mathcal{t}^{-1}\delta^{- 1} \mathcal{t}) \varepsilon \delta\mathcal{u}) \\
    & = & \varepsilon^{- 1} (\mathcal{u}^{-1} (\mathcal{t}^{-1}\delta^{- 1} \mathcal{t} \varepsilon \delta)  \mathcal{u}).
  \end{eqnarray*}
  Now $\delta \prec \mathcal{t}$ so $\mathcal{t}^{-1}\delta^{- 1} \mathcal{t} \prec \mathcal{t}$ by Lemma~\ref{lem-small-conjugate}, so $\mathcal{t}^{-1}\delta^{- 1} \mathcal{t} \varepsilon \delta \prec \mathcal{t}$  by
  Proposition~\ref{prop-growth-order}(\ref{def-growth-order-prod}), so $\mathcal{u}^{-1} (\mathcal{t}^{-1}\delta^{- 1} \mathcal{t} \varepsilon \delta)  \mathcal{u} \prec \mathcal{t}$ by Lemma~\ref{lem-small-conjugate}, whence finally $[f,g] \prec \mathcal{t} \asymp f$ by
  Proposition~\ref{prop-growth-order}(\ref{def-growth-order-prod}).
\end{proof}

\begin{proposition}
  \label{prop-central-commute}If $\mathcal{s} \in \mathcal{G}^>$
  is scaling, then the centraliser of each $f \asymp \mathcal{s}$ is
  commutative.
\end{proposition}

\begin{proof}
  Let $f \asymp \mathcal{s}$ and let $g, h \in \mathcal{C} (f)$. Assume for
  contradiction that $[g, h] \neq 1$. Then, since $[g, h] \in \mathcal{C}
  (f)$, we have $[g, h] \asymp f$. This contradicts
  \Cref{lem-small-commutator}.
\end{proof}

\subsection{Growth order groups}\label{subsection-growth-order-groups}

Given an ordered group $(\mathcal{G}, \cdot, 1, <)$, we consider the following
axiomatic property:

\begin{descriptioncompact}
  \item[$\textbf{GOG3}$] \label{gog3}$\mathcal{G}$ has scaling elements.
\end{descriptioncompact}

Using the $\mathcal{L}_{\text{og}}$-definable abbreviations $\sim$ and $\asymp$, a natural first-order formulation of \hyperref[gog3]{\textbf{GOG3}} is $\forall a \exists b \forall c\exists d  (a \neq 1 \rightarrow ((c \asymp a) \rightarrow ([d,b] = 1 \wedge c \sim d)))$. 

\begin{definition}
  \label{def-growth-order-group}We say that an ordered group $(\mathcal{G},
  \cdot, 1, <)$ is a {\tmem{{\tmstrong{growth order group}}}}{\index{growth
  order group}} if it satisfies \hyperref[gog1]{\textbf{GOG1}}, \hyperref[gog2]{\textbf{GOG2}} and \hyperref[gog3]{\textbf{GOG3}}.
\end{definition}

 All Abelian ordered groups are growth order groups. We write $T_{\tmop{gog}}$\label{autolab13} for the
$\mathcal{L}_{\tmop{og}}$-theory $T_{\tmop{og}} \cup \{ \hyperref[gog1]{\textbf{GOG1}},
\hyperref[gog2]{\textbf{GOG2}}, \hyperref[gog3]{\textbf{GOG3}} \}$. A CT-group is a group in which
  centralisers of non-trivial elements are Abelian. As corollaries of \Cref{prop-central-commute},
we have:

\begin{corollary}
  \label{cor-CT}Growth order groups are CT-groups.
\end{corollary}

\begin{corollary}
  Any non-Abelian growth order group has trivial center.
\end{corollary}

\subsection{Skeletons}\label{subsection-centraliser-relations} Let $\mathcal{G}$ be a growth order group. Fix a $\rho \in v (\mathcal{G})$ and consider the set\label{autolab14}
\[ \mathcal{C}_{\rho} \assign \{\tmop{res} (f) \suchthat v (f) = \rho
   \vee f = 1\} . \]
Recall that $(\mathcal{C}_{\rho}, \lessdot)$ is linearly ordered. For all
$\tmop{res} (f), \tmop{res} (g) \in \mathcal{C}_{\rho}$, we set $\tmop{res} (f) +
\tmop{res} (g) \assign \tmop{res} (fg)$ if $fg \asymp f$ and $\tmop{res} (f) +
\tmop{res} (g) \assign \tmop{res} (1)$ if $fg \prec f$.

\begin{lemma}
  The structure $(\mathcal{C}_{\rho}, +, \tmop{res} (1), \lessdot)$ is an
  ordered Abelian group. Moreover, given a scaling element $\mathcal{s}$ in
  $\mathcal{G}$, the function $\varphi_{\mathcal{s}} \of \mathcal{C}
  (\mathcal{s}) \longrightarrow \mathcal{C}_{v (\mathcal{s})} \; ; f
  \mapsto \tmop{res} (f)$ is an isomorphism.
\end{lemma}

\begin{proof}
  The operation $\mathord{+} \of \mathcal{C}_{\rho} \times \mathcal{C}_{\rho}
  \longrightarrow \mathcal{C}_{\rho}$ is well-defined. For $\tmop{res} (h) \in
  \mathcal{C}_{\rho}$ where $v (h) = \rho$, since $\mathcal{s}$ is
  scaling, there is a unique $f \in \mathcal{C} (\mathcal{s})$ with $h \sim
  f$, whence $\tmop{res} (f) = \tmop{res} (h)$. So $\varphi_{\mathcal{s}}$ is
  surjective. Let $f, g \in \mathcal{C} (\mathcal{s})$. Note that $fg \in \mathcal{C}(\mathcal{s})$, so $fg = 1$ or $fg \asymp \mathcal{s}$. We thus have $fg = 1
  \Longleftrightarrow fg \prec f \Longleftrightarrow v (fg) < \rho$.
  So $\tmop{res} (fg) = \tmop{res} (f) + \tmop{res} (g)$. If $1 < f$, then $1
  \lessdot f$, so $\tmop{res} (1) \lessdot \tmop{res} (f)$. Altogether this
  shows that $\varphi_{\mathcal{s}}$ is an isomorphism between the
  $\mathcal{L}_{\tmop{og}}$-structures $(\mathcal{C}(\mathcal{s}), \cdot, 1,
  <)$ and $(\mathcal{C}_{\rho}, +, \tmop{res} (1), \lessdot)$. In particular,
  the latter is a an ordered Abelian group.
\end{proof}

We call $(\mathcal{C}_{\rho})_{\rho \in v(\mathcal{G})}$ the \textit{skeleton} of $\mathcal{G}$. If $\mathcal{H}$ is a
growth order group, then each ordered group homomorphism $\Phi \of \mathcal{G} \longrightarrow \mathcal{H}$ induces a homomorphism of skeletons, i.e. a nondecreasing map
\[ \Phi_{v} \of v (\mathcal{G}) \longrightarrow v
   (\mathcal{H}) \: ; \: v (g) \longmapsto v (\Phi (g)) \]
and, for each $\rho \in v (\mathcal{G})$, an ordered group homomorphism
\[ \Phi_{\rho} \of \mathcal{C}_{\rho} \longrightarrow
   \mathcal{C}_{\Phi_{v} (\rho)} \hspace{0.8em} ; \: \tmop{res} (f)
   \mapsto \tmop{res} (\Phi (f)) . \]

\subsection{On the structure of growth order
groups}\label{subsection-structure}

Any ordered group is {\cite[Theorem~1]{Iwasawa:free}} a quotient by a convex
normal subgroup of an ordered free group. However that description is far from
being as precise and concrete as the Hahn embedding theorem {\cite{Hahn1907}}
(see {\cite[Theorem~4.C]{Glass}}) for Abelian ordered groups, which construes
them as lexicographically ordered groups of formal commutative series with
real coefficients. We expect that a similar description exists for growth order groups, as we next explain.

Let $\mathcal{G}$ be a growth order group, and let $\mathcal{S}$ be a set of
unique scaling representatives for each valuation. Given
$\mathcal{s} \in \mathcal{S}$ and $c = \tmop{res} (g) \in
\mathcal{C}_{v (\mathcal{s})}^{\neq}$, we let $\mathcal{s}^{[c]}$ denote the unique element of $\mathcal{C} (\mathcal{s})$ with $\mathcal{s}^{[c]} \sim g$. We also write
$\mathcal{s}^{[0]} \assign 1$. Given $f_0 \in \mathcal{G}^{\neq}$, there
are a unique $\mathcal{s}_0 \in \mathcal{S}$ with $\mathcal{s}_0 \asymp f_0$ and
a unique $c_0 \in \mathcal{C}_{v (\mathcal{s}_0)}$ with $f_0 \sim
\mathcal{s}_0^{[c_0]}$. Define
\begin{equation}
  f_1 \assign \mathcal{s}^{[- c_0]} f_0 . \label{eq-choice}
\end{equation}
Reiterating the process for $f_1$ if $f_1 \neq 1$ and continuing further, we
obtain an $\ell \leqslant \omega$, a strictly $\prec$-decreasing sequence
$(\mathcal{s}_n)_{n < \ell}$ in $\mathcal{S}$ and a sequence $(c_n)_{n < \ell}
\in \prod_{n < \ell} \mathcal{C}_{v (\mathcal{s}_n)}$ with
\[ f_0 \approx \mathcal{s}_0^{[c_0]} \mathcal{s}_1^{[c_1]} \mathcal{s}_2^{[c_2]}
   \cdots \mathcal{s}_n^{[c_n]} \cdots, \]
in the sense that $ (\mathcal{s}_0^{[c_0]} \mathcal{s}_1^{[c_1]}
\mathcal{s}_2^{[c_2]} \cdots \mathcal{s}_n^{[c_n]})^{- 1} f_0\prec \mathcal{s}_n$
whenever $n < \ell$. If $\ell = \omega$, then there may exist several elements
of $\mathcal{G}^{\neq}$ with the same expansion as $f_0$ (consider for instance
an ultrapower of $\mathcal{G}$), so describing $f_0$ in full entails
extending this process inductively. This points
to the existence of an embedding of $\mathcal{G}$ into an ordered group of
formal non-commutative series
\begin{equation}
  \mathfrak{g}_0^{[c_0]} \mathfrak{g}_1^{[c_1]} \cdots
  \mathfrak{g}_{\gamma}^{[c_{\gamma}]} \cdots, \gamma < \lambda
  \label{eq-right-trailing}
\end{equation}
where $(c_{\gamma})_{\gamma < \lambda} \in \prod_{\gamma < \lambda}
\mathcal{C}_{\mathfrak{g}_{\gamma}}$, $(\mathfrak{g}_{\gamma})_{\gamma <
\lambda} \in v (\mathcal{G})^{\lambda}$ is strictly decreasing and $\lambda$ is an ordinal. In other words, it is conceivable that
there is a non-commutative version of the Hahn embedding theorem for growth
order groups.

The construction of such an ordered group is
difficult, and it requires additional information besides the skeleton. Moreover, several issues that are
absent in the Abelian case appear here.

First, the choice in \eqref{eq-choice} of expanding $f_0$ systematically on the
right is arbitrary. One could expand $f_0$ on the left, or even alternate
choices. Indeed, given an infinite limit ordinal $\kappa$ and a function $N
\of \kappa \longrightarrow \{\tmop{left}, \tmop{right}\}$, one may expand $f_0$
on the side prescribed by $N (\gamma)$ at each stage $\gamma < \kappa$. This
induces a linear ordering on $\kappa$ which we call tree-like. How can one
describe series with tree-like support?

Secondly, studying examples of groups of transseries shows that in certain
cases, extending $\mathcal{G}$ with transfinite expansions as in
\eqref{eq-right-trailing} forces the existence of valuations that are not
comparable to elements in $\mathcal{G}$. More precisely,
there can be series $s \assign \mathfrak{g}_0^{[c_0]} \mathfrak{g}_1^{[c_1]}
\cdots \mathfrak{g}_{\gamma}^{[c_{\gamma}]} \cdots$ and elements $\mathfrak{g}
\in v (\mathcal{G})$ such that the valuation of $s\mathfrak{g}^{[c]}s^{-
1}$ should lie in an unfilled cut in $(v (\mathcal{G}), \prec)$. So
an embedding theorem must involve constraints on the skeleton of
$\mathcal{G}$.

\begin{question}
  {\tmstrong{Embedding problem.}}\label{question-infinite} For a linearly
  ordered set $(I, <)$ and a family $(\mathcal{C}_i)_{i \in I}$ of Abelian
  ordered groups, under what conditions can one define a group law $\ast$ on
  the set $\mathbf{H}_{i \in I} \mathcal{C}_i$ of functions $f \in \Pi_{i \in
  I} \mathcal{C}_i$ with anti-well-ordered support $\tmop{supp} f = \{i \in I
  \suchthat f (i) \in \mathcal{C}_i \setminus \{0\} \}$, ordered lexicographically, such
  that
  \begin{itemizedot}
    \item $(\mathbf{H}_{i \in I} \mathcal{C}_i, \ast, 1, <)$
    is a growth order group with skeleton
    $\simeq (\mathcal{C}_i)_{i \in I}$,
    
    \item for all growth order groups $\mathcal{G}$ with skeleton
    $\simeq  (\mathcal{C}_i)_{i \in I}$, there is an embedding of ordered
    groups $\mathcal{G} \longrightarrow \mathbf{H}_{i \in I} \mathcal{C}_i$\,?
  \end{itemizedot}
\end{question}

As a first step toward answering this question, we showed \cite{Bag:mg} that certain groups of transseries can be represented as groups $(\mathbf{H}_{i \in I} \mathcal{C}_i, \ast, 1, <)$.

\section{Constructions of growth ordered groups}\label{section-constructions}

We now give methods for constructing growth order groups.

\begin{example}
  \label{ex-hyperseries}We constructed {\cite{Bag:hyperclosed}} an ordered
  field of formal series $\tilde{\mathbb{L}}$ equipped with a composition law
  $\mathord{\circ} \of \tilde{\mathbb{L}} \times
  \tilde{\mathbb{L}}^{>\mathbb{R}} \longrightarrow \tilde{\mathbb{L}}$ and showed {\cite[Propositions~9.23 and~10.25]{Bag:hyperclosed}}
  that $(\tilde{\mathbb{L}}^{>\mathbb{R}}, \circ, <)$ is a growth order group with Archimedean
  centralisers {\cite[Proposition~10.24]{Bag:hyperclosed}}.
\end{example}

\subsection{Semidirect products}\label{subsection-semidirect-products}

Let $(\mathcal{G}, \cdot, 1, <)$, $(G, +, 0, <)$ be ordered groups.
For clarity, we will use additive denotation for $G$, but we do not assume
that $(G, +, 0)$ is Abelian. Let a morphism $\rho \of (\mathcal{G}, \cdot, 1)
\longrightarrow \tmop{Aut} (G, +, 0)$ be given with the following properties:

\begin{descriptioncompact}
  \item[$\textbf{MGA1}$] \label{ga1}Each $\rho (g), g \in \mathcal{G}$ is
  strictly increasing.
  
  \item[$\textbf{MGA2}$] \label{ga2}For $f, g \in \mathcal{G}$ with $f < g$
  and $a \in G^{>}$ we have $\rho (f) (a) < \rho (g) (a)$.
\end{descriptioncompact}

For
$(g, a) \in \mathcal{G} \times G$, we write
\[ g \ast a \assign \rho (g) (a) . \]
We consider the lexicographically ordered semidirect product $\mathcal{G}
\rtimes_{\rho} G$, i.e. the Cartesian product $\mathcal{G} \times G$ equipped
with the operation
\[ \forall (f, a), (g, b) \in \mathcal{G} \times G, (f, a) \cdot (g, b)
   \assign (fg, (f \ast a) + b), \]
and the lexicographic ordering
\[ \forall (g, a), (h, b) \in \mathcal{G} \times G, (g, a) < (h, b)
   \Longleftrightarrow \text{($g < h$ or ($g = h$ and $a < b$)).} \]

   Note that the inverse of an $(f,a) \in \mathcal{G}
\rtimes_{\rho} G$ is given by
\[(f,a)^{-1} = (f^{-1},f^{-1} \ast (-a)).\] 
\begin{proposition}
  \label{prop-gog-product}The structure $(\mathcal{G} \rtimes_{\rho} G, \cdot,
  (1, 0), <)$ is an ordered group, and the functions
  \[ G \longrightarrow \mathcal{G} \rtimes_{\rho} G ; a \mapsto (1, a)
     \text{{\hspace{3em}}and{\hspace{3em}}$\mathcal{G} \longrightarrow
     \mathcal{G} \rtimes_{\rho} G ; f \mapsto (f, 0)$} \]
  are embeddings.
\end{proposition}

\begin{proof}
  The lexicographic ordering is linear, so we need only show that $\mathcal{G}
  \rtimes_{\rho} G$ is an ordered group. Assume that $(g, b) > (h, c)$. If $g
  > h$, then $fg > fh$ and $gf > hf$, so $(f, a) \cdot (g, b) > (f, a) \cdot
  (h, c)$ and$(g, b) \cdot (f, a) > (h, c) \cdot (f, a)$. Otherwise $g = h$
  and $b > c$, so $f \ast b > f \ast c$, whence
  \[ (f, a) \cdot (g, b) = (fh, f \ast b + a) > (fh, f \ast c + a) = (f, a)
     \cdot (h, c) . \]
  Likewise $f \ast a + b > f \ast a + c$ so
  \[ (g, b) \cdot (f, a) = (hf, f \ast a + b) > (hf, f \ast a + c) = (h, c)
     \cdot (f, a) . \]
  This shows that $\mathcal{G} \rtimes_{\rho} G$ is an ordered group. It is easily checked that the two functions above are embeddings.
\end{proof}

We consider two further conditions on $(\mathcal{G}, G)$:

\begin{descriptioncompact}
  \item[{\tmstrong{MGA3}}] \label{ga3}For all $g \in \mathcal{G}$ and $b \in
  G$, for sufficiently large $g' \in \mathcal{C} (g)$, there is a $b' \in G$
  with
  \[ g \ast b' + b = g' \ast b + b' . \]
  \item[$\textbf{MGA4}$] \label{ga4}For all $a \in G$, $f \in \mathcal{G}^{>}$
  and $b \in G^{>}$, we have
  \[ f \ast b > a + b - a. \]
\end{descriptioncompact}

\begin{remark}
  \label{rem-sufficiently-large}Let $(g, b) \in (\mathcal{G} \rtimes_{\rho}
  G)^{\neq}$. For $(f, a) \in \mathcal{G} \rtimes_{\rho} G$, we have $(f, a)
  \cdot (g, b) = (g, b) \cdot (f, a)$ if and only if
  \[ fg = gf \text{{\hspace{3em}}and{\hspace{3em}}$f \ast b + a = g \ast a +
     b.$} \]
  The first condition means that $f \in \mathcal{C} (g)$. Now given $h \in
  \mathcal{C} (g)^{>}$ sufficiently large, by \hyperref[ga3]{\textbf{MGA3}}, there is an $a \in
  G$ with~$g \ast a - a = h \ast b - b$, hence $(h, a) \in \mathcal{C} (g,
  b)$.
\end{remark}

\begin{remark}
  \label{rem-if-Abelian-factor}If $G$ is Abelian, then \hyperref[ga4]{\textbf{MGA4}} follows from
  \hyperref[ga2]{\textbf{MGA2}}.
\end{remark}

\begin{proposition}
  \label{prop-affine-gog}Let $\rho$ satisfy \hyperref[ga1]{$\mathbf{MGA1}$}--\hyperref[ga4]{$\mathbf{MGA4}$}. If $\mathcal{G}$ and $G$ are growth order groups,
  then so is $(\mathcal{G} \rtimes_{\rho} G, \cdot, (1, 0), <)$.
\end{proposition}

\begin{proof}
  We first prove \hyperref[gog1]{\textbf{GOG1}}. Let $(f, a), (g, b) \in \mathcal{G}^{>}$ with
  $(f, a) > (g, b)$, and let $(g', b') \in \mathcal{C} (g, b)$. Assume first
  that $g = 1$, so $b > 1$. Assume for contradiction that $g' > 1$. then we
  have $(g', g' \ast b + b') = (g', b') \cdot (1, b) = (1, b) \cdot (g', b') =
  (g', b' + b)$, so $g' \ast b = b' + b - b'$. But this contradicts
  \hyperref[ga4]{\textbf{MGA4}}. If $g' < 1$, then $(f, a) > (g', b')$. So we may assume that $g'
  = 1$. If $f > 1$, then $(f, a) > (g', b')$. Otherwise, we must have $f = 1$
  and thus $a > b$. Now \hyperref[gog1]{\textbf{GOG1}} in $G$ gives an $a' \in \mathcal{C} (a)$
  with $a' > b'$, whence $(1, a') \in \mathcal{C} (1, a)$ and $(1, a') > (1,
  b')$. We now treat the case when $g > 1$. We have $g' \in \mathcal{C} (g)$ where
  $f \geqslant g$, so by \hyperref[gog1]{\textbf{GOG1}} in $\mathcal{G}$ there is an $f' \in
  \mathcal{C} (f)$ with $f' \geqslant g'$. In view of
  \Cref{rem-sufficiently-large}, we may choose $f'$ sufficiently large so that
  $f' > g'$ and that there be an $a' \in G$ with $(f', a') \in \mathcal{C} (f,
  a)$. We have $(f', a') > (g', b')$, hence \hyperref[gog1]{\textbf{GOG1}} holds in $\mathcal{G}
  \rtimes_{\rho} G$.
  
  Let $(g, b), (f, a) \in \mathcal{G} \rtimes_{\rho} G$ with $(f, a)
  >\mathcal{C} (g, b)$ and $(g, b) > (1, 0)$. Assume that $g = 1$, so $b > 0$
  and $f \geqslant 1$. We have
  \begin{eqnarray*}
    (f, a) \cdot (1, b) \cdot (f, a)^{- 1} & = & (f, a) \cdot (1, b) \cdot
    (f^{- 1}, f^{- 1} \ast (- a))\\
    & = & (1, f \ast (f^{- 1} \ast (- a) + b) + a)\\
    & = & (1, (- a) + f \ast b + a) .
  \end{eqnarray*}
  If $f = 1$, then the condition $(1, a) >\mathcal{C} (1, b)$ amounts to $a
  >\mathcal{C} (b)$, \ so \hyperref[gog2]{\textbf{GOG2}} in $G$ gives $(- a) + b + a > b$. That
  is,
  \[ (f, a) \cdot (1, b) \cdot (f, a)^{- 1} > (1, b) . \]
  If $f > 1$, then \hyperref[ga4]{\textbf{MGA4}} gives $(- a) + f \ast b + a > b$, whence again
  $(f, a) \cdot (1, b) \cdot (f, a)^{- 1} > (1, b)$.
  
  Assume now that $g > 1$. So we must have $f >\mathcal{C} (g)$, whence
  $fgf^{- 1} > g$ by \hyperref[gog2]{\textbf{GOG2}} in $\mathcal{G}$. This implies that $(f, a)
  \cdot (g, b) \cdot (f, a)^{- 1} > (g, b)$. Therefore \hyperref[gog2]{\textbf{GOG2}} is
  satisfied.
  
  We now prove \hyperref[gog3]{\textbf{GOG3}}. Let $(g, b) \in \mathcal{G} \rtimes_{\rho} G$ with
  $(g, b) \neq (1, 0)$ and let $(f, a) \asymp (g, b)$. If $g = 1$, then we
  must have $f = 1$ and $a \asymp b$ in $G$. Given a scaling element $\mathcal{s}$ in $G$ with $\mathcal{s} \asymp b$, we see that $(1,
  \mathcal{s})$ is scaling in $\mathcal{G} \rtimes_{\rho} G$ with $(1,
  \mathcal{s}) \asymp (f, a)$. If $g \neq 1$, then we must have $f \asymp g$.
  Let $\mathcal{t} \in \mathcal{G}$ be scaling with $\mathcal{t} \asymp g$
  and let $\mathcal{u} \in \mathcal{C} (\mathcal{t})$ with $\mathcal{u} \sim
  f$. Then $(\mathcal{u}, 1) \sim (f, a)$ in $\mathcal{G} \rtimes_{\rho} G$,
  which implies that $(\mathcal{t}, 1)$ is scaling. So \hyperref[gog3]{\textbf{GOG3}} holds in
  $\mathcal{G} \rtimes_{\rho} G$.
  
  This shows that $\mathcal{G} \rtimes_{\rho} G$ is a growth order group.
\end{proof}

\begin{example}
  \label{ex-affine}{\tmstrong{Positive affine maps.}} Consider an ordered
  field $K$ and an ordered vector space $\left( G, +, 0,
  <, \mathord{.} \right)$ over $K$. The ordered groups $(K^{>},
  \cdot,1, <)$ and $(G, +, 0, <)$ are growth order groups, as they are Abelian. We have an action $\rho$ of $K^{>}$ on $G$ by scalar multiplication. That is $\rho (\lambda) (a) \assign \lambda \mathbin{.} a$ for all $\lambda \in K^{>}$ and $a \in G$. Then $K^{>} \rtimes_{\rho} G$ is
  naturally isomorphic to the group of strictly increasing affine functions $K
  \longrightarrow G \: ; x \mapsto \lambda \mathbin{.} x + a$ for $(\lambda,
  a) \in K^{>} \times G$, under composition, and where the ordering is given
  by
  \[ \left( x \mapsto \lambda \mathbin{.} x + a \right) > x \text{\ \ iff
     \ \ $\lambda \mathbin{.} b + a > b$ for sufficiently
     large $b \in G$.} \]
  The axioms \hyperref[ga1]{\textbf{MGA1}} and \hyperref[ga2]{\textbf{MGA2}} follow from the fact that $(G, +, 0, <,
  .)$ is an ordered vector space over $K$. We write $\tmop{Aff}_K^+ (G)$ for
  the ordered group $K^{>} \rtimes_{\rho} G$ given by \Cref{prop-gog-product}. Since $(G, +, 0)$ is Abelian, the axiom \hyperref[ga4]{\textbf{MGA4}} is satisfied. Lastly,
  given $\lambda, \lambda' \in K^{>}$ with $\lambda > 1$ and $a \in G$, we
  have $\rho (\lambda) (a) - a = \rho (\lambda - 1) (a)$ so $\rho (\lambda')
  (b) = \rho (\lambda) (a) - a$ for $b \assign \frac{\lambda - 1}{\lambda'}
  \mathbin{.} a \in G$. In particular \hyperref[ga3]{\textbf{MGA3}} holds. Therefore
  $\tmop{Aff}_K^+ (G)$ is a growth order group.
\end{example}

\subsection{Quotients}\label{subsection-quotients}

Given an ordered group $(\mathcal{G},
\cdot, 1, <)$ and a normal and convex subgroup $N \trianglelefteqslant
\mathcal{G}$, the quotient $\mathcal{G}/ N$ is an ordered group  (see {\cite[Section~1.4]{Fuchs11}} or {\cite[{\textsection}4
(1--2) p 260]{Levi:ord-grp}}) for the relation
\begin{equation}
  gN < hN \Longleftrightarrow g < h. \label{eq-quotient-ordering}
\end{equation}

\begin{lemma}
  {\tmem{\label{lem-quotient-map-isotonic}{\cite[{\textsection}4 (1) p
  260]{Levi:ord-grp}}}} The quotient map $\mathcal{G} \longrightarrow
  \mathcal{G}/ N$ is an ordered group homomorphism.
\end{lemma}

The ordering on $\mathcal{G}$ is lexicographic with respect to the
orderings on $\mathcal{G}/ N$ and $N$. That is, we have
\begin{equation}
  \mathcal{G}^{>} = \{ g \in \mathcal{G} \suchthat \nobracket (gN > N) \vee (g
  \in N^{>} \nobracket) \} . \label{eq-lexico-quotient}
\end{equation}

When the short exact sequence $0 \rightarrow N \rightarrow \mathcal{G}
\rightarrow \mathcal{G}/ N \rightarrow 0$ splits, and given a complement $H$
of $N$ in $\mathcal{G}$, we have an ordered group isomorphism $\mathcal{G}
\simeq \mathcal{G}/ N \rtimes_{\rho} N$ for the morphism $\rho \of
\mathcal{G}/ N \longrightarrow \tmop{Aut} (N)$ given by
\[ \forall g \in \mathcal{G}, \forall f \in N, \rho (gN) (f) \assign hfh^{- 1}
\]
for the unique $h \in H \cap gN$, and where $\mathcal{G}/ N \rtimes_{\rho} N$
is lexicographically ordered.

We shall now adapt these ideas to the case of growth order groups. If we want both
$N$ and $\mathcal{G}/ N$ to be growth ordered groups, we have to impose
further conditions on $(\mathcal{G}, N)$. This leads to the following
definition:

\begin{definition}
  Let $\mathcal{G}$ be a growth order group. A $\preccurlyeq$-initial subgroup
  of $\mathcal{G}$ is a non-empty subset $N \subseteq \mathcal{G}$ such that
  for all $f \in N$ and $g \in \mathcal{G}$, we have $g \preccurlyeq f
  \Longrightarrow g \in N$.
\end{definition}

That an $\preccurlyeq$-initial subgroup is indeed a subgroup follows from
Proposition~\ref{prop-growth-order}(\ref{def-growth-order-inv},
\ref{def-growth-order-prod}). For the sequel of
\Cref{subsection-quotients}, we fix a growth order group $(\mathcal{G}, \cdot,
1, <)$ and a normal and $\preccurlyeq$-initial subgroup~$N \subseteq
\mathcal{G}$.

\begin{proposition}
  \label{prop-growth-initial-quotient}Let $H\subseteq \mathcal{G}$ be a $\preccurlyeq$-initial subgroup. Then $H$ is a growth
  order group which is convex in $\mathcal{G}$. 
\end{proposition}

\begin{proof}
  That $H$ is convex follows from \Cref{prop-convex-asymp}. We note by $\preccurlyeq$-initiality that the centraliser in $H$ of an $h \in H$ is simply its centraliser in $\mathcal{G}$. This is easily seen to imply that $H$ is a growth order group.
\end{proof}

\begin{proposition}
  \label{prop-quotient-gog}Assume that the following holds
  \begin{equation}
    \forall f, g \in \mathcal{G} \setminus N, [f, g] \in N \Longrightarrow f
    \asymp g. \label{eq-quotient-centralizer-condition}
  \end{equation}
  Then $\mathcal{G}/ N$, with the ordering given by
  {\tmem{\eqref{eq-quotient-ordering}}}, is a growth order group.
\end{proposition}

\begin{proof}
  Let $f, g \in \mathcal{G}$ with $fN > gN > N$. In particular $f, g \in
  \mathcal{G}^{>}$ and $f > g$. Let $g_0 N \in \mathcal{C} (gN)$, so
  $[g_0, g] \in N$. We have $g_0 \asymp g$ by
  \eqref{eq-quotient-centralizer-condition}. \hyperref[gog1]{\textbf{GOG1}} in $\mathcal{G}$,
  gives an $f_0 \in \mathcal{C} (f)$ with $f_0 \geqslant g_0$, hence $f_0 N \geqslant fN$. We have $[f_0, f] = 1 \in N$ so $f_0 N
  \in \mathcal{C} (fN)$. This shows that \hyperref[gog1]{\textbf{GOG1}} holds in $\mathcal{G}/
  N$.
  
  We next derive \hyperref[gog2]{\textbf{GOG2}}. Let $f, g > N$ with $(fN) \succ (gN)$. We have
  $\mathcal{C} (g) N \subseteq \mathcal{C} (gN)$, so $fN >\mathcal{C} (g) N$,
  which is equivalent to $f >\mathcal{C} (g) N$. In particular, we have $f
  >\mathcal{C} (g)$, so $f^{- 1} \succ g^{- 1}$. By
  \eqref{eq-quotient-centralizer-condition}, we obtain $[f^{- 1}, g^{- 1}] \nin
  N$. But $[f^{- 1}, g^{- 1}] > 1$ by \hyperref[gog2]{\textbf{GOG2}} in $\mathcal{G}$, so $[f^{-
  1}, g^{- 1}] > N$. That is, we have $fgf^{- 1} N > gN$, whence \hyperref[gog2]{\textbf{GOG2}}
  holds in $\mathcal{G}/ N$.
  
  Finally, let $g \in \mathcal{G} \setminus N$. Let $\mathcal{s}$ be scaling in $\mathcal{G}$ with $\mathcal{s} \asymp g$, and let $f \in
  \mathcal{G}$ with $(fN) \asymp (gN)$ in $\mathcal{G}/ N$. From
  \eqref{eq-quotient-centralizer-condition}, we deduce that there are $g', g''
  \in \mathcal{C} (g)$ with $g' \leqslant f \leqslant g''$. This implies that
  $f \asymp g$, so there is a $\mathcal{t} \in \mathcal{C} (\mathcal{s})$ with
  $\mathcal{t} \sim f$. We have $\mathcal{t}N \in \mathcal{C} (\mathcal{s}N)$
  and $(fN)  (\mathcal{t}N)^{- 1} = (f\mathcal{t}^{- 1}) N \prec fN$. We claim
  that $\mathcal{C} (\mathcal{s}N) = \{ \mathcal{u}N \suchthat \mathcal{u} \in
  \mathcal{C} (\mathcal{s}) \}$. Indeed, let $g \in \mathcal{G}$ with $gN \in
  \mathcal{C} (\mathcal{s}N)$, so $[g, \mathcal{s}] \in N$. We have $g \asymp
  \mathcal{s}$ so there is a $\mathcal{t} \in \mathcal{C} (\mathcal{s})$ with
  $g \sim \mathcal{t}$. Writing $\delta \assign \mathcal{t}^{- 1} g$, we have
  $[g, \mathcal{s}] = \delta^{- 1} \mathcal{t}^{- 1} \mathcal{s}^{- 1}
  \mathcal{t} \delta \mathcal{s}= [\mathcal{t}\delta, \mathcal{s}]$. Since $\mathcal{t}$ and $\mathcal{s}$ commute, we obtain $[g,\mathcal{s}]= [\delta, \mathcal{s}] \in N$. As $\delta
  \prec \mathcal{s}$, we deduce with \eqref{eq-quotient-centralizer-condition}
  that $\delta \in N$, so $gN =\mathcal{t}N$ as claimed. Recall by Corollary~\ref{cor-CT} that $\mathcal{C}(\mathcal{s})$ is Abelian. Thus
  $\{ \mathcal{u}N \suchthat \mathcal{u} \in \mathcal{C} (\mathcal{s}) \}$ is
  Abelian, and $\mathcal{s}N$ is scaling in $\mathcal{G}/ N$. So \hyperref[gog3]{\textbf{GOG3}} holds.
\end{proof}

Given two linearly ordered sets $(A, <)$ and $(B, <)$, we write $A \amalg B$
for the disjoint union $A \times \{0\} \sqcup B \times \{1\}$ ordered so that
$A \times \{0\} < B \times \{1\}$ and that $a \mapsto (a, 0)$ and $b \mapsto
(b, 1)$ are ordered embeddings $A \longrightarrow A \amalg B$ and $B
\longrightarrow A \amalg B$ respectively. In the next proof, we use the notation $v_{\mathcal{G}}$ for the standard valuation on a growth order group $\mathcal{G}$, in order to distinguish between various growth order groups.

\begin{proposition}
  \label{prop-quotient-ranks}Assume that
  {\tmem{\eqref{eq-quotient-centralizer-condition}}} holds. We have an
  isomorphism of ordered sets
  \[ \Phi : v_{\mathcal{G}} (\mathcal{G}) \longrightarrow
     v_N (N) \amalg v_{\mathcal{G}/ N} (\mathcal{G}/ N)
  \]
  defined by $\Phi (v (g)) \assign
  v (gN)$ if $g \nin N$ and $\Phi (v(g)) = v_N (g)$ if $g \in N^{\neq}$.
\end{proposition}

\begin{proof}
  Let $g, h \in \mathcal{G}^{\neq}$
  with $g \asymp h$. If $g \nin N$, then we have $g \succ N$, and $hN \asymp
  gN$ by \Cref{lem-quotient-map-isotonic}, whence $\Phi
  (v_{\mathcal{G}} (g))$ is well defined. If $g \in N$, then $h \in
  N$. Since $N$ is $\preccurlyeq$-initial, we have $v_N (g) = v_{\mathcal{G}} (g) =
  v_{\mathcal{G}} (h) = v_N (h)$, so $\Phi$ is
  well-defined. It is clear that $\Phi$ is surjective.
  
  Now let $f, g \in \mathcal{G}$ with $f \prec g$. We want to prove that $\Phi
  (v_{\mathcal{G}} (f)) < \Phi (v_{\mathcal{G}} (g))$. If $f
  \nin N$ and $g \nin N$, then $v (fN) < v (gN)$ by
  \Cref{lem-quotient-map-isotonic}. So $\Phi (v_{\mathcal{G}} (f)) =
  v_{\mathcal{G}/ N} (f) < v_{\mathcal{G}/ N} (g) = \Phi
  (v_{\mathcal{G}} (g))$. If $f \in N$ and $g \nin N$, then $f \prec
  g$ and $\Phi (v_{\mathcal{G}} (f)) < \Phi (v_{\mathcal{G}}
  (g))$ by definition. If $f, g \in N$, then $\Phi (v_{\mathcal{G}}
  (f)) = v_N (f) < v_N (g) = \Phi (v_{\mathcal{G}}
  (g))$. This concludes the proof.
\end{proof}

\subsection{Growth order groups of finite value set}\label{subsection-finite-rank}

We fix a non-trivial growth order group $\mathcal{G}$ such that $v
(\mathcal{G})$ has a maximal element $v (f_0)$. Let $\mathcal{s}$ be scaling with $\mathcal{s} \asymp f_0$. Write
$\mathcal{G}^{\prec \mathcal{s}}\assign \{ g \in \mathcal{G} \suchthat g \prec \mathcal{s}\}$. Note that $\mathcal{G}^{\prec \mathcal{s}}\neq \varnothing$.

\begin{proposition}
  The set $\mathcal{G}^{\prec \mathcal{s}}$ is a normal and
  $\preccurlyeq$-initial subgroup of $\mathcal{G}$.
\end{proposition}

\begin{proof}
  This set is
  $\preccurlyeq$-initial by definition. It is normal by \Cref{lem-small-conjugate}. 
\end{proof}

By \Cref{prop-growth-initial-quotient}, the subgroup $\mathcal{G}^{\prec \mathcal{s}}$ is a growth order group.

\begin{proposition}
  The subgroup $\mathcal{C} (\mathcal{s})$ is a complement of
  $\mathcal{G}^{\prec \mathcal{s}}$.
\end{proposition}

\begin{proof}
  We have $\mathcal{G}^{\prec \mathcal{s}} \cap \mathcal{C} (\mathcal{s}) = \{ \mathcal{t} \in
  \mathcal{C} (\mathcal{s}) \suchthat \mathcal{t} \prec \mathcal{s} \} = \{ 1 \}$. For
  $g \in \mathcal{G}$, we either have $g \prec \mathcal{s}$, and then $g \in
  \mathcal{G}^{\prec \mathcal{s}}$, or $g \asymp \mathcal{s}$, and then given
  $\mathcal{t} \in \mathcal{C} (\mathcal{s})$ with $\mathcal{t} \sim g$, we
  have $g\mathcal{t}^{- 1} \prec \mathcal{s}$, whence $g = (g\mathcal{t}^{-
  1}) \mathcal{t} \in \mathcal{G}^{\prec \mathcal{s}} \mathcal{C}
  (\mathcal{s})$.
\end{proof}

Thus the sequence $0 \rightarrow \mathcal{G}^{\prec \mathcal{s}} \rightarrow
\mathcal{G} \rightarrow \mathcal{G}/\mathcal{G}^{\prec \mathcal{s}}
\rightarrow 0$ splits, and we have a natural isomorphism $\mathcal{G}^{\prec
\mathcal{s}} \rtimes \mathcal{C}
(\mathcal{s}) \longrightarrow \mathcal{G}$. If follows by induction that if
$v (\mathcal{G}) = \{\rho_1, \ldots, \rho_n \}$ is
finite with $\rho_1 > \cdots >\rho_n$, then $\mathcal{G}$ is
an iterated semidirect product
\begin{equation}
  \mathcal{G} \simeq (\cdots \: (\mathcal{C}_{\rho_n} \rtimes
  \mathcal{C}_{\rho_{n - 1}}) \rtimes \cdots)
  \rtimes \mathcal{C}_{\rho_1} .
  \label{eq-finite-growth-rank}
\end{equation}
This can be taken as a conclusion to our discussion in
\Cref{subsection-structure} in the case of finite value set, i.e. a positive
answer to Question~\ref{question-infinite} in that case.

\begin{proposition}
  Suppose that $\mathcal{G}$ has value set $n > 0$ and
  let $\mathcal{t}$ be scaling with $v
  (\mathcal{t}) = \min v (\mathcal{G}^{\neq})$. If $\mathcal{C}
  (\mathcal{t})$ is Archimedean, then $n \leqslant 2$.
\end{proposition}

\begin{proof}
  Assume for contradiction that $n > 2$. Using the above  decomposition $n - 3$ times, we may assume that $n = 3$. Fix two scaling elements
  $\mathcal{s}_1, \mathcal{s}_2$ with
  $\mathcal{t} \prec \mathcal{s}_1 \prec \mathcal{s}_2$. So $\mathcal{G} \simeq \mathcal{G}_1
  \rtimes \mathcal{C} (\mathcal{s}_2)$
  where $\mathcal{G}_1 =\mathcal{C} (\mathcal{t}) \rtimes \mathcal{C} (\mathcal{s}_1).$
  Let $\sigma \in
  \tmop{Aut} (\mathcal{G}_1)$ be the conjugation by $\mathcal{s}_2$ and let $\chi \in \tmop{Aut}(\mathcal{C}(\mathcal{t}))$ be the conjugation by
  $\mathcal{s}_1$. Since $\mathcal{t} \prec
  \mathcal{s}_1$ in $\mathcal{G}_1$, we have $\sigma (\mathcal{t}) \prec
  \sigma ( \mathcal{s}_1)$, whence $\sigma (\mathcal{t}) \asymp
  \mathcal{t}$. But then $\sigma (\mathcal{t}) \in \mathcal{C}
  (\mathcal{t})$. For $n \in \mathbb{N}$ we have $\mathcal{s}_1^n \prec \mathcal{s}_2$, so $\sigma
  (\mathcal{t}) > \chi^{[n]} (\mathcal{t})$ by \hyperref[gog2]{\textbf{GOG2}}. Since
  $\mathcal{C} (\mathcal{t})$ is Archimedean, this contradicts
  {\cite[Theorem~1.5.1]{MuRhem}}.
\end{proof}

\subsection{O-minimal germs}\label{subsection-o-minimal}

Let $\mathcal{M}= (M, \ldots)$ be a first-order structure in a language
$\mathcal{L}$. Assume that $\mathcal{M}$ has definable Skolem functions
(allowing parameters). This is the case for instance if $\mathcal{M}$ is an
o-minimal expansion of an ordered group in a language expanding
$\mathcal{L}_{\tmop{og}}$.

Let $n > 0$ and let $p$ be an $n$-type in $\mathcal{M}$ over $M$ whose finite subsets are realised in $\mathcal{M}$. Let $p
(\mathcal{M}) \assign \{\varphi (M^n) \suchthat \varphi \in p\}$ be the
corresponding ultrafilter on the Boolean algebra of definable subsets of
$M^n$. Consider the set $\mathcal{F}_n$ of functions $M^n \longrightarrow M$
that are definable in $\mathcal{M}$ with parameters, and the set
$\mathcal{M}_p$ of germs at $p$
\[ [f]_p \assign \{g \in \mathcal{F}_n \suchthat \exists X \in p (\mathcal{M}), f \text{ and $g$ coincide on $X$} \} \]
of such functions. If $R$ is a relation symbol of arity $k \in \mathbb{N}$ in
the corresponding language (including function symbols and
constant symbols), then $R$ is interpreted on $\mathcal{M}_p$ as the well-defined
subset of tuples $([f_1], \ldots, [f_k])$ for which there is an $X \in
   p (\mathcal{M})$ with $\mathcal{M} \vDash R \left[
   f_1 \left( \overline{m} \right), \ldots, f_n \left( \overline{m} \right)
   \right]$ for all $\overline{m} \in X$.

It is a folklore result that $\mathcal{M}_p$ is an elementary extension of
$\mathcal{M}$ for the natural inclusion
$\Psi \of \mathcal{M} \longrightarrow \mathcal{M}_p$ sending $\overline{m_0}
\in M$ to the germ of the constant function $\overline{m} \mapsto
\overline{m_0}$. This follows from the following lemma:

\begin{lemma}
  \label{lem-los}For all $\mathcal{L}$-formulas $\varphi (v_1, \ldots, v_k)$ with parameters in $M$
  and $f_1, \ldots, f_k \in \mathcal{F}_n$, we have $\{ \overline{m} \in M^n \suchthat \mathcal{M} \vDash \varphi
     (f_1 (\overline{m}), \ldots, f_p (\overline{m}))\} \in p (\mathcal{M})$ if and only if $\mathcal{M}_p \vDash \varphi ([f_1], \ldots, [f_k])$.
     
\end{lemma}

Suppose that $\mathcal{L}$ contains a binary relation symbol $<$ and that
$\mathcal{M}= (M, <, \ldots)$ is o-minimal. The set of formulas
$m < v_0$, in one free variable $v_0$, where $m$ ranges in $M$ induces a unique type
$p_{\infty}$ over $M$ called the type at infinity. The germ $[f]$ at $p_{\infty}$ of an $f
\in \mathcal{F}_n$ is simply its germ at $+ \infty$. We write
$\mathcal{M}_{\infty} \assign \mathcal{M}_{p_{\infty}}$\label{autolab15}. The
ordering on $\mathcal{M}_{\infty}$ is given by $[f_0] < [f_1]
\Longleftrightarrow f_0 (m) < f_1 (m)$ for all sufficiently large $m \in M$. By the monotonicity theorem {\cite[Chapter~3, (1.2)]{vdD98}} a definable function $f \of M \longrightarrow
M$ is strictly monotonic on some neighborhood of $+\infty$. A germ $[f]$ lies above each $m \in M$ under the embedding
$\mathcal{M} \longrightarrow \mathcal{M}_{\infty}$ if and only if $f$ tends to
$+ \infty$ at $+ \infty$. We define
$\mathcal{G}_{\mathcal{M}}$\label{autolab16} as the subset of
$\mathcal{M}_{\infty}$ of germs $[f]$ with~$[f] > M$. A germ in
$\mathcal{G}_{\mathcal{M}}$ cannot be constant or strictly decreasing, so it
is strictly increasing. We write $\tmop{id}$ for the identity function on $M$,
so $[\tmop{id}] \in \mathcal{G}_{\mathcal{M}}$.

Since $\mathcal{M}$ is o-minimal, for any $[f], [g] \in \mathcal{F}_{\infty}$,
there is an $m \in M$ such that $f ((m, + \infty))$ is a neighbourhood of~$+
\infty$. We may choose $m$ so that $f ((m, + \infty)) = (f (m), + \infty)$. So $f$ induces a strictly increasing bijection between two neighbourhoods of
$+ \infty$. The germ of $f \circ g$ lies in $\mathcal{G}_{\mathcal{M}}$.
Since this germ does not depend on $f, g$ we may define $[f] \circ [g] \assign
[f \circ g]$. Note that $[f] \circ [\tmop{id}] = [\tmop{id}] \circ [f] = [f]$.
Writing $f^{\tmop{inv}}$ for the inverse of $f \of (m, + \infty)
\longrightarrow (f (m), + \infty)$, we see that $[f^{\tmop{inv}}]$ only
depends on $[f]$, and we have $[f] \circ [f^{\tmop{inv}}] = [f^{\tmop{inv}}]
\circ [f] = [\tmop{id}]$. Thus $(\mathcal{G}_{\mathcal{M}}, \circ,
[\tmop{id}])$ is a group. The ordering on $\mathcal{G}_{\mathcal{M}}$
induced by that on~$\mathcal{M}_{\infty}$ is a left-ordering because the
germs are strictly increasing. It is a right-ordering by definition. So
$(\mathcal{G}_{\mathcal{M}}, \circ, [\tmop{id}], <)$ is an ordered group.

This raises the naive question: is $\mathcal{G}_{\mathcal{M}}$ always a growth
order group? The answer is negative. Indeed, it is known
{\cite[Theorem~8]{BaiBalVer}} that given any ordered group
$(\mathcal{G}, \cdot, 1, <)$, the structure $\mathcal{M} \assign (\mathcal{G},
<, (t_g)_{g \in \mathcal{G}})$ where each $t_g$ for $g \in \mathcal{G}$ is the
unary function $\mathcal{G} \longrightarrow \mathcal{G} \: ; h \mapsto gh$
eliminates quantifiers and has a universal axiomatisation. In particular, it
is o-minimal, and $g \mapsto [t_g]$ is an isomorphism between $(\mathcal{G},
\cdot, 1, <)$ and $(\mathcal{G}_{\mathcal{M}}, \circ, [\tmop{id}], <)$. If
$\mathcal{G}$ is not a growth order group, then neither is
$\mathcal{G}_{\mathcal{M}}$. We may still ask whether $\mathcal{G}_{\mathcal{M}}$ is a growth order group when $\mathcal{M}$ expands the real ordered field. We will answer this question in the positive in a particular case in the next
section. We finish with a positive answer to the naive question
for pure ordered groups:

\begin{example}
  Let $\mathcal{M} \assign (G, +, 0, <)$ be a non-trivial o-minimal ordered
  group. This is a divisible, Abelian ordered group {\cite{PilStein:omin}}, so it has Skolem functions. Recall {\cite{Rob:complete}} that the
  $\mathcal{L}_{\tmop{og}}$-theory $T_{\tmop{daog}}$ of non-trivial divisible
  Abelian ordered group is complete and has quantifier elimination in
  $\mathcal{L}_{\tmop{og}}$. It has a universal
  axiomatisation in the language $\mathcal{L}_{\tmop{doag}} \assign
  \langle \cdot, 1, <, \tmop{Inv}, (\mu_q)_{q \in \mathbb{Q}} \rangle$ where
  each $\mu_q, q \in \mathbb{Q}$ is interpreted as the
  scalar multiplication $x \mapsto q \mathbin{.} x$. This implies that the
  germ at $+ \infty$ of each definable function $G \longrightarrow G$ is that
  of a term in $\mathcal{L}_{\tmop{doag}}$. So each element of $\mathcal{M}_{\infty}$ is the germ of
  \[ G \longrightarrow G \: ; x \mapsto q \mathbin{.} x + y \]
  for fixed $q \in \mathbb{Q}$ and $y \in G$. In other words $\mathcal{G}_{\mathcal{M}}$ is isomorphic to the growth order
  group $\tmop{Aff}_{\mathbb{Q}}^+ (G)$ of \Cref{ex-affine}.
\end{example}

\section{H-fields with composition and inversion}\label{section-H-fields}

An {\tmem{H-field}}{\index{H-field}} {\cite{AvdD02,AvdD03}} is an ordered valued field $(K, +, \times,
0, 1, <, \mathcal{O})$ with convex valuation ring $\mathcal{O}$ and maximal ideal
thereof $\mathcal{o}$, equipped with a derivation $\partial \of K
\longrightarrow K$ such that the following conditions are satisfied:

\begin{descriptioncompact}
  \item[$\textbf{HF1}$] \label{hf1}$\forall a \in \mathcal{O}, \exists c \in
  \tmop{Ker} (\partial), a - c \in \mathcal{o}.$
  
  \item[$\textbf{HF2}$] \label{hf2}$\forall a \in K, a > \tmop{Ker} (\partial)
  \Longrightarrow \partial (a) > 0$.
\end{descriptioncompact}

We usually denote $\tmop{Ker} (\partial)$ by $C$. This is a
subfield of $K$ called the {\tmem{field of constants}}{\index{field of
constants}}. We write $K^{> C} \assign \{ a \in K \suchthat \forall c \in C,a>c\}$. For $a\in K$, we often write $a'=\partial(a)$, and we use the Landau notations 
$\mathcal{O} (a) \assign \mathcal{O}a = \{\delta a \suchthat \delta \in \mathcal{O}\}$ and
$\mathcal{o} (a) \assign \mathcal{o}a = \{\varepsilon a \suchthat \varepsilon \in \mathcal{o}\}$. So
$\mathcal{O} (1) =\mathcal{O}$ and $\mathcal{o} (1) =\mathcal{o}$. For $a \in K^{\times}$, we write\label{autolab17}
\[ a^{\dag} \assign \frac{a'}{a} \in K. \]
Note that $(ab)^{\dag} = a^{\dag} + b^{\dag}$ and $(ca)^{\dag} = a^{\dag}$ for
all $b \in K^{\times}$ and $c \in C^{\times}$. We have the following important
valuative inequality {\cite[Lemma~1.1]{AvdD02}}:
\begin{equation}
  \forall a, b \in \mathcal{o}, b' \in \mathcal{o} (a^{\dag}) . \label{eq-gap}
\end{equation}
Furthermore, we have {\cite[Corollary 1]{Rosli80}} l'Hospital's rule 
\begin{equation}
  \forall f, g \in \mathcal{H}, ((f \in \mathcal{o}(g) \wedge g \nin \Theta
  (1)) \Longrightarrow f' \in \mathcal{o}(g')). \label{eq-d-valued}
\end{equation}
\subsection{H-fields with
composition}\label{subsection-H-fields-with-composition}

We now expand H-fields with a composition law.

\begin{definition}
  \label{def-H-field-with-composition}An {\tmstrong{{\tmem{H-field with
  composition}}}}{\index{H-field with composition}} (over $C = \tmop{Ker}
  (\partial)$) is an H-field $(K, +, \cdot, 0, 1, <, \mathcal{O}, \partial)$
  with a fixed $x \in K^{> C}$ such that $x' = 1$, and a binary operation
  $\mathord{\circ} \of K \times K^{> C} \longrightarrow K$ satisfying the
  following conditions:
  
  \begin{descriptioncompact}
    \item[$\textbf{HFC1}$] \label{hfc1}For all $b \in K^{> C}$, the function
    $K \longrightarrow K \: ; \: a \mapsto a \circ b$ is a $C$-linear morphism
    of ordered rings.
    
    \item[$\textbf{HFC2}$] \label{hfc2}For all $a \in K$ and $b, d \in
    K^{> C}$, we have $a \circ (b \circ d) = (a \circ b) \circ d$.
    
    \item[$\textbf{HFC3}$] \label{hfc3}For all $a \in K^{> C}$, the function
    $K^{> C} \longrightarrow K^{> C} \: ; \: b \mapsto a \circ b$ is strictly
    increasing.
    
    \item[$\textbf{HFC4}$] \label{hfc4}For all $a \in K$ and $b \in K^{> C}$,
    we have
    \[ a \circ x = a \text{{\hspace{3em}}and{\hspace{3em}}$x \circ b = b.$} \]
    \item[$\textbf{HFC5}$] \label{hfc5}Let $a, \delta \in K$ and $b \in K^{>
    C}$ with $\delta \in \mathcal{o} (b)$ and $(a^{\dag} \circ b) \delta \in
    \mathcal{o}$. For all $n \in \mathbb{N}$, we have
    \[ a \circ (b + \delta) - \sum_{k \leqslant n} \frac{a^{(k)} \circ b}{k!}
       \delta^k \in \mathcal{o} ((a^{(n)} \circ b) \delta^n), \]
    where $a^{(k)}$ denotes the $k$-th derivative of $a$.
  \end{descriptioncompact}
\end{definition}

Consider the language $\mathcal{L}_{\tmop{hfc}}$ expanding the language of
ordered valued differential fields with a constant symbol $\underline{x}$ and a binary function
symbol $\circ$. We interpret $\underline{x}$ on $K$ as expected and extend
$\circ$ to $K \times K$ by setting $a \circ b \assign 0$ if $b \nin K^{> C}$.
Thus $K$ is an $\mathcal{L}_{\tmop{hfc}}$-structure, and the class of H-fields
with composition is elementary in $\mathcal{L}_{\tmop{hfc}}$.

The axioms \hyperref[hfc1]{\textbf{HFC1}}--\hyperref[hfc4]{\textbf{HFC4}} imply that $(K^{> C}, \circ, x, <)$ is an
ordered monoid that acts by automorphisms on
$(K, +, \cdot, 0, 1, <, \mathcal{O})$, by post-composition. In order to avoid confusion between
compositions and products in $K$, given an $a \in K^{> C}$ and an $n \in
\mathbb{N}$, we write $a^{[n]}$\label{autolab18} for the $n$-fold iterate of
$a$ (i.e. its $n$-th power in the monoid $K^{> C}$). If $a$ has an inverse in
$K^{> C}$, then we denote it by $a^{\tmop{inv}}$\label{autolab19} and we set
$a^{[- n]} \assign (a^{\tmop{inv}})^{[n]} = (a^{[n]})^{\tmop{inv}}$.

\begin{example}
  \label{ex-fractions}Let $C$ be an ordered field. Let $C (x)$ be a purely
  transcendental simple extension, ordered so that $x > C$. Write $\mathcal{O}$ for the convex hull of $C$ in $C (x)$, which
  is the set of fractions with degree $\leq 0$.
  
  We have a derivation $\partial \of C (x) \longrightarrow C (x)$
  with respect to $x$, which is determined by $C = \tmop{Ker} (\partial)$ and
  $\partial (x) = 1$. And $(C (x), +, \cdot, 0, 1, \partial,
  \mathcal{O}, <)$ is an H-field. For $P \in C (x)$ and $Q \in C (x)^{> C}$,
  since $Q$ lies above each pole of $P$, the compositum $P \circ Q$ is
  well-defined. It is easy to see that \hyperref[hfc1]{\textbf{HFC1}}--\hyperref[hfc4]{\textbf{HFC4}} are satisfied.
  Les us now justify that \hyperref[hfc5]{\textbf{HFC5}} holds. Let $F \in C(x)$ and $b,\delta,n$ as in \hyperref[hfc5]{\textbf{HFC5}}. We have $F' \in \mathcal{O} x^{-1} F$, so $(F^{(k+1)} \circ b) \delta^{k+1} \in \mathcal{O} (F^{(k)}  \circ b) \delta^k \frac{\delta}{b} \subseteq \mathcal{o} (F^{(k)}  \circ b) \delta^k$ for each $k \in \mathbb{N}$. We have formal identity $F\circ(b+y) = \sum \limits_{k \in \mathbb{N}} \frac{F^{(k)} \circ b}{k!} y^k$ in $C[[x,y]]$, and the previous argument entails that plugging $\delta$ for $y$ gives a convergent sum for the valuation topology on $C[[x]]$. It also entails that $F\circ(b+\delta)-\sum \limits_{k \leqslant n} \frac{F^{(k)} \circ b}{k!} \delta^k = \sum \limits_{k>n} \frac{F^{(k)} \circ b}{k!} \delta^k \in \mathcal{O}((F^{(n+1)} \circ b) \delta^{n+1}) \subseteq\mathcal{o}((F^{(n)} \circ b) \delta^n)$.
  
  Note that each H-field
  with composition over $C$ contains $C (x)$ as an
  $\mathcal{L}_{\tmop{hfc}}$-substructure.
\end{example}

\begin{example}
  \label{ex-grid-based-transseries}Consider the field $\mathbb{T}_g$ of
  grid-based transseries {\cite{Ec92,vdH:phd}}. We have a derivation and
  composition law {\cite{vdH:ln}} on $\mathbb{T}_g$ such that it is an H-field
  with field of constants $\mathbb{R}$ and that \hyperref[hfc1]{\textbf{HFC1}}, \hyperref[hfc2]{\textbf{HFC2}},
  \hyperref[hfc4]{\textbf{HFC4}} and \hyperref[hfc5]{\textbf{HFC5}} are satisfied. As for \hyperref[hfc3]{\textbf{HFC3}}, it follows
  from the inclusion of $\mathbb{T}_g$ in the field of finitely nested
  hyperseries of {\cite{Bag:hyperclosed}}, where it holds. By
  {\cite[Section~5.4]{vdH:ln}}, this field has inversion.
\end{example}

We will see other, more analytic examples in the next section (see
\Cref{cor-Hardy-H}). We now state a few simple consequences of the axioms.

\begin{remark}
  If $\varepsilon \in \mathcal{o}$, then $\varepsilon' \in \mathcal{o} ((x^{-
  1})^{\dag}) =\mathcal{o} (x^{- 1})$ by \eqref{eq-gap}. In particular
  $\varepsilon' \in \mathcal{o}$, so the derivation on $K$ is small as per
  {\cite[p 7]{vdH:mt}}.
\end{remark}

As an ordered field, any H-field has a field topology, called the order topology, for which the family of $(-
\varepsilon, \varepsilon), \varepsilon \in K^{>}$ is a fundamental system of
neighbourhoods of $0$. We understand limits in that sense.

\begin{lemma}
  Let $K$ be an H-field with composition. For $a \in K$ and $b \in K^{> C}$,
  we have
  \[ a' \circ b = \lim_{\underset{\delta \neq 0}{\delta \rightarrow 0}} 
     \frac{a \circ (b + \delta) - a \circ b}{\delta}. \]
\end{lemma}

\begin{proof}
  Let $\delta \in K$ be sufficiently small in absolute value, so that $\delta
  \in \mathcal{o} (b)$ and $(a^{\dag} \circ b) \delta \in \mathcal{o}$. By
  \hyperref[hfc5]{\textbf{HFC5}} for $n = 1$, we have $a \circ (b + \delta) - a \circ b - (a'
  \circ b) \delta \in \mathcal{o} ((a'' \circ b) \delta^2)$, so
  \[ \left| \frac{a \circ (b + \delta) - a \circ b}{\delta} - (a' \circ b)
     \delta \right| < | (a'' \circ b) \delta |. \]
  Letting $\delta$ tend to $0$, we obtain the desired result.
\end{proof}

\begin{lemma}
  \label{lem-chain-rule}Let $K$ be an H-field with composition. Let $a \in K$
  and $b \in K^{> C}$. We have
  \[ (a \circ b)' = (a' \circ b) b' . \]
\end{lemma}

\begin{proof}
  Write $\tau(\delta) \assign \delta^{-1} ((a \circ b) \circ (x + \delta) - a \circ b)$ for all $\delta \neq 0$, so
    \[(a \circ b)' = \lim_{\underset{\delta \neq 0}{\delta \rightarrow 0}} 
    \tau(\delta).\]
  By \hyperref[hfc5]{\textbf{HFC5}} for $(b, x, \delta)$, we can have $b \circ (x + \delta) - b \in b' \delta +\mathcal{O} (b'' \delta^2)$
 arbitrarily small by choosing $\delta$ small enough. In turn, applying \hyperref[hfc5]{\textbf{HFC5}} for $(a, b, b
  \circ (x + \delta) - b)$ we obtain $\delta \tau(\delta) - (a' \circ b) b'
     \delta \in \mathcal{O} ((a' \circ b) b'' \delta^2)$
  provided $\delta$ is sufficiently small. We thus have $\tau(\delta) - (a' \circ b) b' \in
     \mathcal{O} ((a' \circ b) b'' \delta)$, hence the result.\end{proof}

We say that $K$ is an {\tmem{H-field with composition and
inversion}}{\index{H-field with composition and inversion}} if furthermore
$(K^{> C}, \circ, x)$ is a group. Then in view of \hyperref[hfc1]{\textbf{HFC1}}--\hyperref[hfc4]{\textbf{HFC4}},
the structure $(K^{> C}, \circ, x, <)$ is an ordered group. We will give
conditions for it to be a growth order group. More precisely, consider the
following conditions on an ordered pair $(\mathcal{G}_0, \mathcal{G}_1)$ of
subgroups of $K^{> C}$:

\begin{descriptioncompact}
  \item[$\mathbf{(\star)}$] \label{dec}The subset $\mathcal{G}_0 \subseteq K^{> C}$ is a normal
  convex subgroup of $K^{> C}$ containing $x + C$, the subset $\mathcal{G}_1 \subseteq K^{> C}$ is
  a complement of $\mathcal{G}_0$ in $K^{> C}$ which is a
  growth order group with Archimedean centralisers, and $\{a \circ
  (a^{\tmop{inv}} + 1) \suchthat a \in \mathcal{G}_1 \}$ is cofinal in
  $\mathcal{G}_0$.
\end{descriptioncompact}

We will obtain \Cref{th-main} as a consequence of the following theorem.

\begin{theorem}
  \label{th-levelled}Let $(K, +, \cdot, 0, 1, \partial, \mathcal{O}, <, \circ,
  x)$ be an H-field over $\mathbb{R}$ with composition and inversion and let $(\mathcal{G}_0, \mathcal{G}_1)$ be as in \hyperref[dec]{$(\mathbf{\star})$}. Then
  $K^{>\mathbb{R}}$ is a growth order group with Archimedean
  centralisers, and $\mathcal{G}_0$ is a growth order
  group which is $\preccurlyeq$-initial in $K^{>\mathbb{R}}$.
\end{theorem}

This will be proved in \Cref{subsection-proofofmain} below.

\subsection{Taylor approximations in Hardy fields}\label{subsection-Taylor}

Let $\mathcal{C}^{< \infty}$ denote the set of all germs $[f]$ at $+ \infty$
of real-valued functions $f$ defined on positive half-lines $(a, + \infty), a
\in \mathbb{R}$ such that for each $k \in \mathbb{N}$, there is
a positive half-line on which $f$ is $k$-times differentiable. We identify
constants with the germs of the corresponding constant functions. Then
$\mathcal{C}^{< \infty}$ is an $\mathbb{R}$-algebra under pointwise sum and
product. Moreover, it is equipped with a partial $\mathbb{R}$-algebra ordering
given by $[f] < [g]$ if and only $f (t) < g (t)$ for all sufficiently
large $t \in \mathbb{R}$. It is a differential ring under derivation of germs
$[f]' \assign [f']$ whenever $f \of (a, + \infty) \longrightarrow \mathbb{R}$
is differentiable. Finally, if $[g] >\mathbb{R}$ in $\mathcal{C}^{< \infty}$,
i.e. if $g$ tends to $+ \infty$ at $+ \infty$, then for all $[f] \in
\mathcal{C}^{< \infty}$ where $f \circ g$ is defined on a positive half-line,
the germ $[f] \circ [g]
\assign [f \circ g]$ only depends on $[f]$ and $[g]$.

We will identify germs with given representatives, trying not to confuse the
reader in the process. Given a germ $g \in \mathcal{C}^{< \infty}$, we write
\begin{eqnarray*}
  \mathcal{o} (g) & \assign & \{f \in \mathcal{C}^{< \infty} \suchthat \forall
  r \in \mathbb{R}^{>}, |f| < r|g|\}\\
  \mathcal{O} (g) & \assign & \{f \in \mathcal{C}^{< \infty} \suchthat \exists
  r \in \mathbb{R}, |f| < r|g|\} \text{, \quad and}\\
  \Theta (g) & \assign & \{f \in \mathcal{C}^{< \infty} \suchthat f \in
  \mathcal{O}(g) \wedge g \in \mathcal{O}(f)\} .
\end{eqnarray*}
We simply write $\mathcal{o}, \mathcal{O}$ and $\Theta$ for $\mathcal{o} (1),
\mathcal{O} (1)$ and $\Theta (1)$ respectively.

Recall that a {\tmem{Hardy field}}{\index{Hardy field}} is a differential subfield of
$\mathcal{C}^{< \infty}$ containing all
constant germs. The induced ordering on such fields is linear {\cite[2 p
107]{Bou07}}.

\begin{definition}
  A {\tmem{{\tmstrong{Hardy field with composition}}}}{\index{Hardy field with
  composition}} is a Hardy field $\mathcal{H}$ which is closed under
  composition of germs. We say that it has inversion if
  $\mathcal{H}^{>\mathbb{R}}$ is closed under inversion.
\end{definition}

\begin{example}
  If $\mathcal{R}$ is an o-minimal expansion of the real ordered field, then
  $\mathcal{R}_{\infty}$ is a Hardy field {\cite[Section~7.1]{vdD98}} with
  composition and inversion.
\end{example}

\begin{example}
  The intersection of all
  $\subseteq$-maximal Hardy fields is a Hardy field with composition
  {\cite{Bosh82}}. It is unknown whether it has inversion.
\end{example}

We will show that certain Hardy fields with composition and inversion are
H-fields with composition and inversion. This mainly entails deriving the
Taylor axiom \hyperref[hfc5]{\textbf{HFC5}} in those fields. If $\mathcal{H}$ is a Hardy field,
then $\mathcal{O} (1) \cap \mathcal{H}$ is a valuation ring on $\mathcal{H}$
for which it is an H-field. The notations above are consistent with that
introduced for H-fields. The derivation on $\mathcal{H}$ is \tmem{small}, i.e.
\begin{equation}
  \mathcal{o}' \subseteq \mathcal{o} \label{eq-small-der} .
\end{equation}
See \cite[Section~2]{Rosli83}} or {\cite[Proposition~9.1.9]{vdH:mt}} for proofs. Toward proving \hyperref[hfc5]{\textbf{HFC5}},
we need a mean value theorem for germs.

\begin{lemma}
  \label{lem-mean-value-Hardy}Let $\mathcal{H}$ be a Hardy field with
  composition and inversion and let $f \in \mathcal{H}$ and $g, h \in
  \mathcal{H}^{>}$ with $g < h$. There is a $c \in \mathcal{H}$ with $g < c <
  h$ and $f \circ h - f \circ g = (h - g) f' \circ c$.
\end{lemma}

\begin{proof}
  Assume first that $f' \in \mathbb{R}$. So $f$ is the germ of an affine
  function $f = a \tmop{id} + b$, and we have $f \circ h - f \circ g = (h - g)
  a = (h - g) f' \circ c$ where $c \assign \frac{g + h}{2} \in (g, h)$.
  
  Assume now that $f' \nin \mathbb{R}$. So $f'$ is the germ of a strictly
  monotonic function. Let $t \in \mathbb{R}$ be large enough so that $h (s) >
  g (s)$ for all $s \geqslant t$, that $f$ is differentiable on $[t, +
  \infty)$ and that $f'$ is strictly monotonic on $[t, + \infty)$. The mean
  value theorem for $f$ gives $f (h (t)) - f (g (t)) = (h (t) - g (t)) f'
  (c_t)$ for a certain $c_t \in (g (t), h (t))$. Since $f'$ is strictly
  monotonic on $[t, + \infty)$, the number $c_t$ is unique, and we have a
  function $t \mapsto c_t$ whose germ $c$ satisfies $c \in (g, h)$ and $f
  \circ h - f \circ g = (h - g) f' \circ c$.
  
  Note that $f' \circ c = \frac{f \circ h - f \circ g}{h - g} \in
  \mathcal{H}$. Our hypothesis that $f' \nin \mathbb{R}$ means that $f'$ is the germ of a strictly monotonic function, which thus induces a bijection $\varphi \of (t_0, + \infty)
  \longrightarrow (t_1, t_2)$ for some $t_0 \geqslant t$ and $t_1, t_2 \in
  \mathbb{R} \cup \{ \pm \infty \}$ with $t_1 < t_2$. By considering
  translations, homotheties and inversions if necessary, we may assume that
  $t_2 = + \infty$, so $c = \varphi^{\tmop{inv}} \circ \frac{f \circ h - f
  \circ g}{h - g}$, lies in $\mathcal{H}$.
\end{proof}

\begin{lemma}
  \label{lem-Tay-alt1}For all $f \in \mathcal{H}^{>\mathbb{R}}$ with $f^{\dag}
  \in \mathcal{O} (\tmop{id}^{- 1})$, we have $(f')^{\dag} \in \mathcal{O}
  (\tmop{id}^{- 1})$.
\end{lemma}

\begin{proof}
  We have $f' \tmop{id} \in \mathcal{O} (f)$ where $f \nin \Theta (1)$, so $(f''
  \tmop{id} + f') \in \mathcal{O} (f')$ by \eqref{eq-d-valued}. We recall that
  $\mathcal{O}$ is a valuation ring on $\mathcal{H}$. Since $f' \in
  \mathcal{O} (f')$, we must have $f'' \tmop{id} \in \mathcal{O} (f')$, i.e.
  $(f')^{\dag} \in \mathcal{O} (\tmop{id}^{- 1})$.
\end{proof}

\begin{lemma}
  \label{lem-Tay-alt2}For all $f \in \mathcal{H}^{>\mathbb{R}}$ with $f^{\dag}
  \nin \mathcal{O} (\tmop{id}^{- 1})$, we have $(f')^{\dag} \in \mathcal{O}
  (f^{\dag})$.
\end{lemma}

\begin{proof}
  By {\cite[Theorem~2]{Rosli83}}, there is a Hardy field $\mathcal{H}^{\ast}$
  containing $\mathcal{H}$ and which is closed under composition on the left
  of strictly positive germs with the germ $\log$ of the natural logarithm.
  Note that $f^{\dag} = (\log \circ f)'$. We have $\tmop{id}^{- 1} \in
  \mathcal{o} (f^{\dag})$ in $\mathcal{H}^{\ast}$ i.e. $\log' \in \mathcal{o}
  ((\log \circ f)')$. So \eqref{eq-d-valued} gives $\log \in \mathcal{o} (\log
  \circ f)$. This means that $\mathbb{N} \log < \log \circ f$, so
  $\tmop{id}^{\mathbb{N}} < f$. In particular $\tmop{id}^2 \in \mathcal{o}
  (f)$ so \eqref{eq-d-valued} yields $2 \tmop{id} \in \mathcal{o} (f')$, whence
  $2 < f''$ by \hyperref[hf2]{\textbf{HF2}}.
  
  Now $- \frac{f'}{f^2} = (f^{- 1})' \in \mathcal{o}$ by \eqref{eq-small-der},
  which means that $f' \in \mathcal{o} (f^2)$. We deduce with
  {\cite[Lemma~1.4]{AvdD02}} that $(f')^{\dag} < (f^2)^{\dag} = 2 f^{\dag}$.
  Since $f'' > 0$ and $f' > 0$, we have $(f')^{\dag} > 0$, so this entails
  that $(f')^{\dag} \in \mathcal{O} (f^{\dag})$.
\end{proof}

\begin{proposition}
  \label{prop-Hardy-conj2-equivalnce}Let $\mathcal{H}$ be a Hardy field with
  composition. Then $\mathcal{H}$ satisfies \hyperref[hfc5]{\textbf{HFC5}} if and only if for all
  $f, g \in \mathcal{H}^{>\mathbb{R}}$ and $\delta \in \mathcal{o} (g)$ with
  $(f^{\dag} \circ g) \delta \in \mathcal{o}$, we have
  \begin{equation}
    f \circ (g + \delta) \in \Theta (f \circ g) . \label{eq-Hfwc2-weak}
  \end{equation}
\end{proposition}

\begin{proof}
  The relation in \eqref{eq-Hfwc2-weak} is implied by \hyperref[hfc5]{\textbf{HFC5}} at $n =
  0$. Assume that \eqref{eq-Hfwc2-weak} holds. Let $f,g,\delta$ be as in the statement of the proposition. We claim that
  \begin{equation}
    \forall n \in \mathbb{N}, (f^{(n + 1)} \circ g) \delta^{n + 1} \in
    \mathcal{o} ((f^{(n)} \circ g) \delta^n) \label{eq-fn+1} .
  \end{equation}
  Indeed, for $n = 0$, this follows   from the assumption on $\delta$. Let $n
  \in \mathbb{N}$ such that \eqref{eq-fn+1} holds at $n$. Suppose that
  $(f^{(n)})^{\dag} \in \mathcal{O} (\tmop{id}^{- 1})$. Then $(f^{(n +
  1)})^{\dag} \in \mathcal{O} (\tmop{id}^{- 1})$ by \Cref{lem-Tay-alt1}, i.e.
  $f^{(n + 2)} \in \mathcal{O} (\tmop{id}^{- 1} f^{(n + 1)})$. Composing with $g$ and then multiplying by $\delta^{n+2}$, we obtain
  \[ (f^{(n + 2)} \circ g) \delta^{n + 2} \in \frac{\delta}{g} \mathcal{O}
     ((f^{(n + 1)} \circ g) \delta^{n + 1}) . \]
  But $\delta \in \mathcal{o} (g)$ so $(f^{(n + 2)} \circ g) \delta^{n + 2}
  \in \mathcal{o} ((f^{(n + 1)} \circ g) \delta^{n + 1})$ as claimed. Suppose
  now that $(f^{(n)})^{\dag} \nin \mathcal{O} (\tmop{id}^{- 1})$. Then
  \Cref{lem-Tay-alt2} gives $(f^{(n + 1)})^{\dag} \in \mathcal{O}
  ((f^{(n)})^{\dag})$ so $((f^{(n + 1)})^{\dag} \circ g) \delta \in
  \mathcal{O} (((f^{(n)})^{\dag} \circ g) \delta) \subseteq \mathcal{o}$ by
  the induction hypothesis. Therefore \ $(f^{(n + 2)} \circ g) \delta
  \delta^{n + 1} \in \mathcal{o} ((f^{(n + 1)} \circ g) \delta^{n + 1})$. We
  conclude by induction that \eqref{eq-fn+1} holds.
  
  Let us now derive \hyperref[hfc5]{\textbf{HFC5}} at a given $n\in \mathbb{N}$. Suppose $\delta\geqslant 0$. Let $r_0, r_1 \in \mathbb{R}^{>}$ and let $t_0 \in \mathbb{R}$ be
  large enough so that $f$ is $\mathcal{C}^{n + 1}$ on $[t_0, + \infty)$, that
  $\delta$ is non-negative on $[t_0, + \infty)$, and that $f^{(n + 1)}$ is
  monotonic on $g (t_0, + \infty)$. By \eqref{eq-Hfwc2-weak} for $f^{(n +
  1)}$, we may also assume that
  \[ r_0  |f^{(n + 1)} (g (t)) | \leqslant |f^{(n + 1)} (g (t) + \delta (t)) |
     \leqslant r_1  |f^{(n + 1)} (g (t)) | \]
  for all $t \in (t_0, + \infty)$. By Taylor's theorem, for $t > t_0$, the
  integral
  \[ I (t) \assign \int_{g (t)}^{g (t) + \delta (t)} \frac{(g (t) + \delta (t)
     - s)^n}{n!} f^{(n + 1)} (s) d s \]
  satisfies $f (g (t) + \delta (t)) = \sum_{k = 0}^n \frac{f^{(k)} (g (t))}{k!} \delta
     (t)^k + I (t)$. Now $| I (t) |$ is bounded by the integral
  \[ \int_{g (t)}^{g (t) + \delta (t)} \frac{(g (t) + \delta (t) - s)^n}{n!}
     r_1  |f^{(n + 1)} (g (t)) | d s = \frac{|f^{(n + 1)} (g (t)) |}{(n + 1)
     !} \delta (t)^{n + 1} . \]
  Thus \hyperref[hfc5]{\textbf{HFC5}} at $n$ follows from \eqref{eq-fn+1}. The case when $\delta
  \leqslant 0$ is similar.
\end{proof}

\begin{lemma}\label{lem-somehfcs}
Let $\mathcal{H}$ be a Hardy field with composition and
  inversion. Then for $x=\tmop{id}$, the axioms \hyperref[hfc1]{\textbf{HFC1}}--\hyperref[hfc4]{\textbf{HFC4}} are satisfied.

  \end{lemma}

\begin{proof}
Note that
  $\mathcal{H}$ is an H-field with $x' = 1$. The monotonicity of germs in
  $\mathcal{H}$ yields \hyperref[hfc3]{\textbf{HFC3}}, whereas \hyperref[hfc1]{\textbf{HFC1}}, \hyperref[hfc2]{\textbf{HFC2}} and
  \hyperref[hfc4]{\textbf{HFC4}} are immediate.\end{proof}

\begin{proposition}
  \label{prop-Hardy-partial-conj2}Let $\mathcal{H}$ be a Hardy field with
  composition and inversion. If there is no $f \in \mathcal{H}$
  with $f > \exp^{[n]}$ for all $n \in \mathbb{N}$, then $\mathcal{H}$
  satisfies \hyperref[hfc5]{\textbf{HFC5}}.
\end{proposition}

\begin{proof}
  By {\cite[Theorem~2]{Rosli83}}, there is a Hardy field $\mathcal{H}^{\ast}$
  containing $\mathcal{H}$ and which is closed under $\exp$ and $\log$. We
  will partly work inside $\mathcal{H}^{\ast}$ so that we may compare our germs $f
  \in \mathcal{H}$ with elements of the form $\exp^{[n]}, n \in \mathbb{Z}$.
  Let $f, g \in \mathcal{H}^{>\mathbb{R}}$ and $\delta \in \mathcal{H} \cap \mathcal{o}(g)$ with
  $(f^{\dag} \circ g) \delta \in \mathcal{o}$. Let $c \in
  \mathcal{C}^{< \infty}$ with $f \circ (g + \delta) - f \circ g = \delta f' \circ c$ as in \Cref{lem-mean-value-Hardy}. We will show that $f \circ (g + \delta)
  \in \Theta (f \circ g)$ by distinguishing two cases.
  
  \textbf{Case 1}: $\exists p>0,f \in \mathcal{O} ((\log)^p)$ in $\mathcal{H}^{\ast}$. Then $f^{\dag} \in \mathcal{O} \left(
    \frac{1}{\tmop{id} \log} \right)$ so $f' \in \mathcal{O} \left(
    \frac{f}{\tmop{id} \log} \right)$. Pick $r \in \mathbb{R}^{>}$ with $|f'| \leqslant
    r\frac{f}{\tmop{id} \log}$. Consider a real number $s \in
    (0,1)$. Recall that $\delta \in \mathcal{o}
    (g)$, so $-s<\frac{c}{g}-1<s$. Lastly $f'$ is the germ of a monotonic function. Combining all this,  for sufficiently large $t > 1$, we have $(1 - s) g (t) < c
    (t) (1 + s)$, and
    \[ | f' (c (t)) | \leqslant \frac{r \max (f (g (t) + \delta (t)), f (g
       (t)))}{(1 - s) g (t) \log (g (t))} . \]
    Since $\delta \in \mathcal{o} (g)$ we get $\frac{| \delta (t) |}{g (t)} <
    \frac{1 - s}{r}$ for sufficiently large $t > 1$. We deduce that
    \[ | \delta (t) f' (c (t)) | \leqslant \frac{\max (f (g (t) + \delta (t)),
       f (g (t)))}{\log (g (t))} . \]
    for sufficiently large $t > 1$. Since $\log \circ g \nin \mathcal{O}$, we
    deduce that $\delta f' \circ c \in \mathcal{o} (f \circ (g + \delta))$ or
    $\delta f' \circ c \in \mathcal{o} (f \circ g)$. In particular $f \circ (g
    + \delta) \in \Theta (f \circ g)$.
    
    \textbf{Case 2}: $\log^{\mathbb{N}} \subseteq \mathcal{O} (f)$ in
    $\mathcal{H}^{\ast}$. Our assumption on $\mathcal{H}$ implies that there
    are an $n \in \mathbb{N}$ and a $p \in \mathbb{N}$ with $\exp^{[n - 2]}
    \in \mathcal{O} (f)$ and $f \in \mathcal{O} ((\exp^{[n - 1]})^p)$ in
    $\mathcal{H}^{\ast}$. We prove by induction on $k \leqslant n$ that for
    all $h \in \mathcal{H}^{\ast}$ with $h >\mathbb{R}$ and $p \in \mathbb{N}$
    with $p > 0$, we have
    \begin{equation}
      (\exp^{[k - 2]} \in \mathcal{O} (h) \wedge h \in \mathcal{O} ((\exp^{[k
      - 1]})^p) \Longrightarrow h \circ (g + \delta) \in \Theta (h \circ g)),
      \label{eq-Hardy-induc}
    \end{equation}
    where $\exp^{[-2]} = \log^{[2]}$ and $\exp^{[-1]} = \log$.
    Note that for $k < n$, $p \in \mathbb{N}^{>}$ and $h \in
    (\mathcal{H}^{\ast})^{>\mathbb{R}}$ with $\exp^{[k - 2]} \in \mathcal{O}
    (h)$ and $h \in \mathcal{O} ((\exp^{[k - 1]})^p)$ we have $h \in
    \mathcal{O} (f)$ so $h^{\dag} \in \mathcal{O} (f^{\dag})$, so
    \begin{equation}
      (h^{\dag} \circ g) \delta \in \mathcal{o} (1) .
      \label{eq-aux-induc-Hardy}
    \end{equation}
    Thus if $k = 0$, then \eqref{eq-Hardy-induc} follows from \textbf{Case
    1}. Let $k < n$ such that \eqref{eq-Hardy-induc} holds at $k$. Let $h \in
    (\mathcal{H}^{\ast})^{>\mathbb{R}}$ and $p \in \mathbb{N}^{>}$ with
    $\exp^{[k - 1]} \in \mathcal{O} (h)$ and $h \in \mathcal{O}
    ((\exp^{[k]})^p)$. We again write $h \circ (g + \delta) - h \circ g =
    \delta h' \circ c$ where $c$ lies strictly between $g$ and $g + \delta$.
    It suffices to show that $\delta h' \circ c \in \mathcal{o} (h \circ g)$
    or that $\delta h' \circ c \in \mathcal{o} (h \circ (g + \delta))$. Note
    that $h' \circ c \in \mathcal{O} (h' \circ g)$ or $h' \circ c \in
    \mathcal{O} (h' \circ (g + \delta))$ by monotonicity of $h'$. If $h' \circ
    c \in \mathcal{O} (h' \circ g)$, then \eqref{eq-aux-induc-Hardy} yields the
    result. So we may assume that $h' \circ c \in \mathcal{O} (h' \circ (g +
    \delta))$. We have $h \in \mathcal{O} ((\exp^{[n]})^p)$, so $\log \circ h
    \in \mathcal{O} (\exp^{[n - 1]})$, so $h^{\dag} \in \mathcal{O} ((\exp^{[n
    - 1]})') \subseteq \mathcal{o} ((\exp^{[n - 1]})^2)$. The induction
    hypothesis at $n - 1$ for $h^{\dag}$ yields $h^{\dag} \circ (g + \delta)
    \in \Theta (h^{\dag} \circ g)$, whence $(h^{\dag} \circ (g + \delta))
    \delta \in \mathcal{o}$, whence $(h' \circ c) \delta \in \mathcal{o} (h
    \circ (g + \delta))$ as desired. By induction, the statement
    \eqref{eq-Hardy-induc} holds for $k = n$, whence in particular $f \circ (g
    + \delta) \in \Theta (f \circ g)$.
  
  We conclude with \Cref{prop-Hardy-conj2-equivalnce} that $\mathcal{H}$
  satisfies \hyperref[hfc5]{\textbf{HFC5}}.
\end{proof}

\begin{corollary}
  \label{cor-Hardy-H}Let $\mathcal{H}$ be a Hardy field with composition and
  inversion. Assume that there is no germ $f \in \mathcal{H}$ with $f >
  \exp^{[n]}$ for all $n \in \mathbb{N}$. Then $\mathcal{H}$ is an H-field
  with composition and inversion.
\end{corollary}

  This result may extend to transexponential Hardy fields with composition and inversion, provided one has some control on the growth of elements of said fields. For instance, we believe it holds in Padgett's transexponential Hardy field with composition \cite{Pad:phd}, provided it has inversion. In general, $\mathcal{H}^{>\mathbb{R}}$  should be contained in a single $T$-level as per \cite{Tyne:phd} (see also \cite{Tyne:height}), for some o-minimal theory $T$.

\begin{corollary}
  \label{cor-elem-Hfield}Let $\mathcal{R}$ be an o-minimal expansion of the
  real ordered field in a first-order language $\mathcal{L}$. Assume that each
  $f \in \mathcal{R}_{\infty}$ lies below a germ $\exp^{[k]}, k \in
  \mathbb{N}$ in $(\mathcal{C}^{< \infty}, <)$. Let $\mathcal{R}^{\ast} =
  (\mathbb{R}^{\ast}, \ldots)$ be an elementary extension of $\mathcal{R}$.
  Consider the ordered field $\mathcal{R}_{\infty}^{\ast}$ with its canonical
  {\tmem{{\cite{vdD98}}}} derivation $\partial$, with the convex hull
  $\mathcal{O}^{\ast}$ of $\mathbb{R}^{\ast}$ as a valuation ring, and
  composition of germs. Then $\mathcal{R}_{\infty}^{\ast}$ is an H-field with
  composition and inversion.
\end{corollary}

\begin{proof}
  The result holds, by \Cref{cor-Hardy-H}, if $\mathcal{R}^{\ast}
  =\mathcal{R}$. The structure $(\mathcal{R}_{\infty}^{\ast}, \partial)$ is a
  differential field by {\cite[Chapter~7, (1.3)]{vdD98}}. Each element $h$ of
  the valuation ring of $\mathcal{R}_{\infty}^{\ast}$ is the germ of a
  definable bounded monotonic function on $\mathbb{R}^{\ast}$, so by
  o-minimality of $\mathcal{R}^{\ast}$, it has a limit $c \in
  \mathbb{R}^{\ast}$. We have $h - c \in \mathcal{o}^{\ast}$ by definition, so
  \hyperref[hf1]{\textbf{HF1}} holds. If $h \in \mathcal{R}_{\infty}^{\ast}$ lies above
  $\mathbb{R}^{\ast}$, then by the monotonicity theorem $h$ must be the germ
  of a strictly increasing function. We deduce with {\cite[Chapter~7, (2.5),
  Lemma~1]{vdD98}} that $h' > 0$. So \hyperref[hf2]{\textbf{HF2}} holds and
  $\mathcal{R}_{\infty}^{\ast}$ is an H-field. Except for \hyperref[hfc1]{\textbf{HFC1}} which
  refers to $\tmop{Ker} (\partial) =\mathbb{R}^{\ast}$, all statements in
  \Cref{def-H-field-with-composition} can be turned, after specialisation of
  the universally quantified variables, into sentences in $\mathcal{L}$. Since
  they hold for $\mathcal{R}_{\infty}$, they hold for
  $\mathcal{R}^{\ast}_{\infty}$. The existence of compositional inverses for
  elements in $\mathcal{G}_{\mathcal{R}^{\ast}}$ has already been established.
  This leaves the axiom \hyperref[hfc1]{\textbf{HFC1}} to justify, but that follows immediately
  from the definition of $\mathbb{R}^{\ast}$.
\end{proof}

\subsection{Conjugacy in H-fields with composition and
inversion}\label{subsection-conjugacy}

We fix an H-field with composition and inversion $(K, +, \cdot, 0, 1,
\mathcal{O}, <, \partial, \circ, x)$ over $\mathbb{R}$ and we write
$\mathcal{G}$ for the group $K^{>\mathbb{R}}$ under composition.

\begin{lemma}
  \label{lem-cof-centralizer}Let $g = x + r_0 + \varepsilon$ where $r_0 \in
  \mathbb{R}$ and $\varepsilon \in \mathcal{o} \cap K^{>}$. Then $\mathcal{C}
  (g)$ is Archimedean and each $h \in \mathcal{C} (g)$ has the form $h = x + r
  + \delta$ for an $r \in \mathbb{R}$ and a $\delta \in \mathcal{o}$.
\end{lemma}

\begin{proof}
  For $n \in \mathbb{Z}$, we claim that $(g^{[n]} - (x + nr_0)) \in
  \mathcal{o}$. Indeed this holds for $n = 0$. Given $n \in \mathbb{Z}$ such
  that $g^{[n]} = x + nr_0 + \varepsilon_n$ where $\varepsilon_n \in
  \mathcal{o}$, we have
  \[ g^{[n + 1]} = x + nr_0 + \varepsilon_n + r_0 + \varepsilon \circ (x +
     nr_0 + \varepsilon_n) = x + (n + 1) r_0 + \varepsilon_{n + 1} \]
  where $\varepsilon_{n + 1} \assign \varepsilon \circ (x + nr_0 +
  \varepsilon_n) \in \mathcal{o}$ by
  \hyperref[hfc1]{\textbf{HFC1}}. So we have the result for all $n \in
  \mathbb{N}$ by induction. Write $\varepsilon_{- 1} \assign g^{[- 1]} - x +
  r_0$. We have
  \[ x = g \circ (x - r_0 + \varepsilon_{- 1}) = x + \varepsilon_{- 1} +
     \varepsilon \circ (x - r_0 + \varepsilon_{- 1}) \]
  where $\varepsilon \circ (x - r_0 + \varepsilon_{- 1}) \in \mathcal{o}$ by
  \hyperref[hfc1]{\textbf{HFC1}}. So
  we must have $\varepsilon_{- 1} \in \mathcal{o}$, and we can use the same
  arguments as in the case $n \geqslant 0$, to show by induction that
  $(g^{[n]} - (x + nr_0)) \in \mathcal{o}$ for all $n \in -\mathbb{N}$.
  
  Now let $h \in \mathcal{C} (g)^{>}$ and assume for contradiction that $\delta
  \assign h - x>\mathbb{R}$. Since $g \circ h = h \circ g$, we have $x + \delta + r_0 + \varepsilon \circ h = x + r_0 +
     \varepsilon + \delta \circ g$.
  So
  \begin{equation}
    \delta + \varepsilon \circ h = \varepsilon + \delta \circ g.
    \label{eq-aux-cof-centralizer}
  \end{equation}
  From $\varepsilon \in \mathcal{o}$ and $\varepsilon > 0$, we deduce by
  \hyperref[hfc3]{\textbf{HFC3}} that $\varepsilon \circ h < \varepsilon$, whereas $\delta \circ
  g > \delta$. This contradicts~\eqref{eq-aux-cof-centralizer}. So $h = x + r +
  \iota$ for a certain $r \in \mathbb{R}$ and a certain $\iota \in
  \mathcal{o}$. Combining these two results, we deduce that $\mathcal{C}
  (g)$ is Archimedean.
\end{proof}

\begin{lemma}
  \label{lem-gog2-R}Let $f, g \in \mathcal{G}^{>}$ with $f > x +\mathbb{R}$,
  and assume that $g = x + 1 + \varepsilon$ for a certain $\varepsilon \in
  \mathcal{o}$ with $\varepsilon > 0$. Then $f \circ g > g \circ f$.
\end{lemma}

\begin{proof}
  Recall that $K$ is an H-field, so $f - x >\mathbb{R}$ entails that $(f - x)'
  > 0$, whence $f' > 1$. We distinguish three cases.
  
  Assume that $f - x \in \mathcal{o} (x)$. So $f = x + \delta$ where $\delta >\mathbb{R}$. We have
  \begin{eqnarray*}
    f \circ g - g \circ f & = & x + 1 + \varepsilon + \delta \circ (x + 1 +
    \varepsilon) - x - \delta - 1 - \varepsilon \circ (x + \delta)\\
    & = & (\delta \circ (x + 1 + \varepsilon) - \delta) + (\varepsilon -
    \varepsilon \circ (x + \delta)) .
  \end{eqnarray*}
  Now $\delta \circ (x + 1 + \varepsilon) - \delta > 0$ because $\delta
  >\mathbb{R}$ and $x + 1 + \varepsilon > x$, and $\varepsilon - \varepsilon
  \circ (x + \delta) > 0$ because $\varepsilon \in K^{>} \cap \mathcal{o}$,
  $\varepsilon > 0$ and $x + \delta > x$. So $f \circ g > g \circ f$ in that
  case.
  
  Assume now that $f \in \Theta (x)$. Then let $r \in \mathbb{R}$ with $f - rx
  \in \mathcal{o} (x)$. So $r>1$. Write $\delta
  \assign f - rx$, so $\delta \in \mathcal{o} (x)$. This time, we have
    \[f \circ g - g \circ f = (r - 1) + (\delta \circ (x + 1 + \varepsilon) - \delta) + (r
    \varepsilon - \varepsilon \circ (x + \delta)) .\]
  
  As in the previous case, the term $\varepsilon - \varepsilon \circ (x +
  \delta)$ is strictly positive. We deduce since $r > 1$ that $r \varepsilon -
  \varepsilon \circ (x + \delta) > 0$. Since $r - 1 > 0$, it suffices to show
  that $\delta \circ (x + 1 + \varepsilon) - \delta \in \mathcal{o}$. This is
  immediate if $\delta \in \mathcal{O}$. Indeed then we find by \hyperref[hf1]{\textbf{HF1}} an
  $r_0 \in \mathbb{R}$ and a $\iota \in \mathcal{o}$ with $\delta = r_0 +
  \iota$. Thus
  \[ \delta \circ (x + 1 + \varepsilon) - \delta \in \Theta (\iota \circ (x +
     1 + \varepsilon) - \iota) \nobracket, \]
  whence $\delta \circ (x + 1 + \varepsilon) - \delta \in \mathcal{o}$ by
  \hyperref[hfc1]{\textbf{HFC1}}.
  
  Assume that $\delta \nin \mathcal{O}$. Recall that $\delta \in \mathcal{o}
  (x)$, so in view of \eqref{eq-d-valued}, we have $\delta' \in \mathcal{o}$.
  Therefore $\delta'  (1 + \varepsilon) \in \mathcal{o}$. By \hyperref[hfc5]{\textbf{HFC5}}, we
  have
  \[ \delta \circ (x + 1 + \varepsilon) - \delta - \delta'  (1 + \varepsilon)
     \in \mathcal{o} (\delta'  (1 + \varepsilon)) . \]
  Since $\delta'  (1 + \varepsilon) \in \mathcal{o}$, we must have $\delta
  \circ (x + 1 + \varepsilon) - \delta \in \mathcal{o}$ as claimed.
  
  We finally treat the remaining case when $f / x >\mathbb{R}$. We have $f >
  x\mathbb{R}^{>}$, so $f^{\tmop{inv}} <\mathbb{R}^{>} x$. Since
  $f^{\tmop{inv}} >\mathbb{R}$, we deduce with \hyperref[hf2]{\textbf{HF2}} that $0 <
  (f^{\tmop{inv}})' <\mathbb{R}^{>}$, i.e. $(f^{\tmop{inv}})' \in
  \mathcal{o}$. It suffices to show that $g^{\tmop{inv}} \circ f^{\tmop{inv}} < f^{\tmop{inv}} \circ
     g^{\tmop{inv}} $.
  Recall as in the proof of Lemma~\ref{lem-cof-centralizer} that $g^{\tmop{inv}} = x - 1 - \delta$ for a certain $\delta \in \mathcal{o}
  \cap K^{>}$. We have $(f^{\tmop{inv}})^{\dag} =
  \frac{(f^{\tmop{inv}})'}{f^{\tmop{inv}}} \in \mathcal{o}$. Since
  $(f^{\tmop{inv}})^{\dag} 1 \in \mathcal{o}$, the axiom \hyperref[hfc5]{\textbf{HFC5}} gives
  \[f^{\tmop{inv}} \circ (x - 1 - \delta) \in f^{\tmop{inv}} -
  (f^{\tmop{inv}})' (1 + \delta) +\mathcal{o} ((f^{\tmop{inv}})').\] Therefore
  $f^{\tmop{inv}} \circ g^{\tmop{inv}} - f^{\tmop{inv}} \in \mathcal{o}$. We
  have
  \[ g^{\tmop{inv}} \circ f^{\tmop{inv}} - f^{\tmop{inv}} = (x - 1 - \delta)
     \circ f^{\tmop{inv}} - f^{\tmop{inv}} = - 1 - \delta \circ f^{\tmop{inv}}
     \in - 1 +\mathcal{o}.\]
  Thus $g^{\tmop{inv}} \circ f^{\tmop{inv}} - f^{\tmop{inv}} < f^{\tmop{inv}}
  \circ g^{\tmop{inv}} - f^{\tmop{inv}}$, so $g^{\tmop{inv}} \circ
  f^{\tmop{inv}} < f^{\tmop{inv}} \circ g^{\tmop{inv}}$.
\end{proof}

We next need to find approximate primitives of elements in $K$. These are large enough that this does not require any further assumption on $K$ (such as having asymptotic integration, see \cite[p 8]{AvdD03}).

\begin{lemma}
  \label{lem-Rosenlicht}Given $\delta \in \mathcal{O} (x^{- 2})$, there is
  an $h \in K$ with $h' - \delta^{-1} \in \mathcal{o} (\delta^{-1}) \cap K^{>}$.
\end{lemma}

\begin{proof}
  In view of {\cite[Theorem~1]{Rosli83}}, it suffices to show that $x^{- 2}
  \in \mathcal{o} (f^{\dag})$ for all $f \in \mathcal{o} \setminus \{0\}$.
  Let~$f \in \mathcal{o} \setminus \{0\}$. By \eqref{eq-gap}, we have $g' \in
  \mathcal{o} (f^{\dag})$ for all $g \in \mathcal{o}$. In particular $(x^{-
  1})' = - x^{- 2} \in \mathcal{o} (f)$, hence the result.
\end{proof}

\begin{lemma}
  \label{lem-x+1}Let $g \in \mathcal{G}^{>}$ be of the form $g = x + \delta$
  where $\delta \in K^{>} \cap \mathcal{O} (x^{- 2})$. There are an $h \in
  \mathcal{G}$ and an $\varepsilon \in \mathcal{o}$ with $\varepsilon > 0$ and
  $h \circ g \circ h^{\tmop{inv}} = x + 1 + \varepsilon$.
\end{lemma}

\begin{proof}
  By \Cref{lem-Rosenlicht}, the condition on $\delta$ implies that there is an
  $h \in K$ such that the germ $\iota \assign h' - \delta^{- 1}$ satisfies
  $\iota \in \mathcal{o} (\delta^{- 1})$ and $\iota > 0$. Since $\delta \in
  K^{>} \cap \mathcal{o}$, the element $\delta^{-1}$ is positive infinite. Note that that $\mathcal{O}' = \mathcal{o}' \subseteq \mathcal{o}$ by (\ref{eq-small-der}), while $f'<0$ for all negative infinite elements by \hyperref[hf1]{$\mathbf{HF1}$}. So $h \in \mathcal{G}$. We have $\delta \circ
  h^{\tmop{inv}} \in \mathcal{o} (h)$ because $\delta \in \mathcal{O}$ whereas
  $h \nin \mathcal{O}$. Finally, we have
  \[ \delta h^{\dag} \in \Theta \left( \frac{\delta}{h \delta} \right)
     \text{\quad and\quad$\frac{1}{h} \in \mathcal{o}$}, \]
  so $\delta h^{\dag} \in \mathcal{o}$. Consider by
 \hyperref[hfc5]{$\mathbf{HFC5}$} the Taylor approximation
  \begin{eqnarray*}
    h \circ g \circ h^{\tmop{inv}} & = & h \circ (h^{\tmop{inv}} + \delta
    \circ h^{\tmop{inv}})\\
    & = & x + (h' \circ h^{\tmop{inv}})  (\delta \circ h^{\tmop{inv}}) +
    \frac{1}{2}  (h'' \circ h^{\tmop{inv}})  (\delta \circ h^{\tmop{inv}})^2 +
    \delta_1
  \end{eqnarray*}
  where $\delta_1 \in \mathcal{o} ((h'' \circ h^{\tmop{inv}})  (\delta \circ
  h^{\tmop{inv}})^2)$. Note that
  \begin{eqnarray*}
    (h' \circ h^{\tmop{inv}})  (\delta \circ h^{\tmop{inv}}) & = & (1 + \iota
    \delta) \circ h^{\tmop{inv}} = 1 + (\iota \delta) \circ h^{\tmop{inv}}
  \end{eqnarray*}
  where $(\iota \delta) \circ h^{\tmop{inv}}$ is positive by
 \hyperref[hfc1]{$\mathbf{HFC1}$}. We have $h'' = (\delta^{-1})' +\iota'$. Now L'Hospital's rule (\ref{eq-d-valued}) entails that the sign of $h''$ is that of $(\delta^{-1})'$, which is positive because $\delta^{-1}
  >\mathbb{R}$. So $h'' \delta^2 > 0$, so $\frac{1}{2}  (h'' \circ
  h^{\tmop{inv}})  (\delta \circ h^{\tmop{inv}})^2 > 0$ whence
  \[ \varepsilon \assign (\iota \delta) \circ h^{\tmop{inv}} + \frac{1}{2} 
     (h'' \circ h^{\tmop{inv}})  (\delta \circ h^{\tmop{inv}})^2 + \delta_1 >
     0. \]
  We have $h \circ g \circ h^{\tmop{inv}} = x + 1 + \varepsilon$ as desired.
\end{proof}

\subsection{Ordered groups in H-fields with composition and
inversion}\label{subsection-proofofmain}

We now prove \Cref{th-levelled}. Let $(K, +, \cdot, 0, 1, \partial,
\mathcal{O}, <, \circ, x)$, $\mathcal{G}_0$ and $\mathcal{G}_1$ be as in the
statement of \Cref{th-levelled}. Consider the projections $\pi_0 \of
K^{>\mathbb{R}} \longrightarrow \mathcal{G}_0$ and $\pi_1 \of K^{>\mathbb{R}}
\longrightarrow \mathcal{G}_1$ with $\pi_0 \pi_1 =
\tmop{Id}_{K^{>\mathbb{R}}}$.

\begin{lemma}
  \label{lem-G0-x+1}For all $g \in \mathcal{G}_0^{>}$, there are a $\varphi
  \in \mathcal{G}_1$ and an $\varepsilon \in K^{>} \cap \mathcal{o}$ with
  \[ \varphi \circ g \circ \varphi^{\tmop{inv}} = x + 1 + \varepsilon . \]
\end{lemma}

\begin{proof}
  Let $\varphi \in \mathcal{G}_1$ such that $\varphi > x^3 + \mathbb{R}$ and $g
  \leqslant \varphi \circ (\varphi^{\tmop{inv}} + 1)$. Thus
  $\varphi^{\tmop{inv}} \circ g \circ \varphi \leqslant x + 1$. We have so $x
  \leqslant \varphi^{[- 3]} \circ g \circ \varphi^{[3]}$ because $g$ is positive in the ordered group $\mathcal{G}_0$. So $x\leqslant \varphi^{[-
  2]} \circ (\varphi^{[2]} + 1)$. Note that $\varphi > x + \mathbb{R}$, so $\varphi' >
  1$ by \hyperref[hf2]{\textbf{HF2}}. We have $(\varphi^{\tmop{inv}})^{\dag} \circ \varphi =
     \frac{1}{\varphi' x} \in \mathcal{o} (1)$.
 \hyperref[hfc5]{\textbf{HFC5}} for $(\varphi^{\tmop{inv}}, \varphi, 1)$ gives
  $\varphi^{\tmop{inv}} \circ (\varphi + 1) - x - (\varphi^{\tmop{inv}})'
  \circ \varphi \in \mathcal{o} ((\varphi^{\tmop{inv}})' \circ \varphi)$,
  so
  \[ \varphi^{\tmop{inv}} \circ (\varphi + 1) - x \in \mathcal{O} \left(
     \frac{1}{\varphi'} \right) . \]
  But $\varphi > x^3 + C$ so $\varphi' > 3 x^2$, so $\delta \assign
  \varphi^{\tmop{inv}} \circ (\varphi + 1) - x$ lies in $\mathcal{O} (x^{-
  2})$. We deduce that \Cref{lem-x+1} applies and yields the result.
\end{proof}

\begin{lemma}
  \label{lem-centraliser-nospill}For $g \in \mathcal{G}_0^{\neq}$, we have
  $\mathcal{C} (g) \subseteq \mathcal{G}_0$.
\end{lemma}

\begin{proof}
  We may assume that $g > x$. Let $\varphi \in \mathcal{G}_1$ be given by Lemma~\ref{lem-G0-x+1} with
  $\varphi^{\tmop{inv}} \circ g \circ \varphi = x + 1 + \varepsilon$ for an
  $\varepsilon \in K^{>} \cap \mathcal{o}$. We have $\mathcal{C} (g) = \varphi
  \circ \mathcal{C} (x + 1 + \varepsilon) \circ \varphi^{\tmop{inv}}$, so it
  suffices to show that $\mathcal{C} (x + 1 + \varepsilon) \subseteq
  \mathcal{G}_0$. This follows from \Cref{lem-cof-centralizer} and the fact
  that $\mathcal{G}_0$ is a convex subgroup of $K^{>\mathbb{R}}$.
\end{proof}

We will use the identity $\mathcal{C} (g) =\mathcal{C} (g) \cap \mathcal{G}_0$
for $g \in \mathcal{G}_0$ without mention.

\begin{corollary}
  The subgroup $\mathcal{G}_0 \subseteq K^{>\mathbb{R}}$ is
  $\preccurlyeq$-initial in $K^{>\mathbb{R}}$.
\end{corollary}

\begin{proposition}
  \label{prop-Archimedean-centralisers-here}The group $K^{>\mathbb{R}}$ has
  Archimedean centralisers. 
\end{proposition}

\begin{proof}
  Let $g \in K^{>\mathbb{R}}$ with $g > x$. If $g \in \mathcal{G}_0$, then by Lemma~\ref{lem-G0-x+1} the ordered group $\mathcal{C} (g)$ is isomorphic to $\mathcal{C} (x + 1 +
  \varepsilon)$ for an $\varepsilon \in K^{>} \cap \mathcal{o}$, whence
  $\mathcal{C} (g)$ is Archimedean by \Cref{lem-cof-centralizer}.
  
  If $g \nin \mathcal{G}_0$, then we must have $g >\mathcal{G}_0$ by
  convexity. For $f, h \in K^{>\mathbb{R}}$, we have $[f, h] = 1
  \Longrightarrow [\pi_1 (f), \pi_1 (h)] = \pi_1 (1) = 1$, so the morphism
  $\pi_1 \upharpoonleft \mathcal{C} (g) \of \mathcal{C} (g) \longrightarrow
  \mathcal{G}_1$ ranges in $\mathcal{C} (\pi_1 (g)) \cap \mathcal{G}_1$. It is
  nondecreasing by \eqref{eq-lexico-quotient}. For $h \in \tmop{Ker} (\pi_1)
  \cap \mathcal{C} (g) =\mathcal{G}_0 \cap \mathcal{C} (g)$, since $g \nin
  \mathcal{G}_0$, we cannot have $h \in \mathcal{G}_0^{\neq}$ by
  \Cref{lem-centraliser-nospill}. Therefore $\pi_1 \upharpoonleft \mathcal{C}
  (g)$ is an embedding of ordered groups $\mathcal{C} (g) \longrightarrow
  \mathcal{C} (\pi_1 (g)) \cap \mathcal{G}_1$. We deduce since its codomain is
  Archimedean that $\mathcal{C} (g)$ is Archimedean.
\end{proof}

\begin{corollary}
  The axiom \hyperref[gog1]{\textbf{GOG1}} holds in $\mathcal{G}_0$ and in $K^{>\mathbb{R}}$.
\end{corollary}

\begin{proof}
  For $K^{>\mathbb{R}}$ this follows from
  Propositions~\ref{prop-Archimedean-centralisers-here}
  and~\ref{prop-Archimedean-gog1}. For $\mathcal{G}_0$, we know by
  Propositions~\ref{prop-Archimedean-centralisers-here} and \Cref{lem-centraliser-nospill} that it has Archimedean centralisers. We conclude with Proposition~\ref{prop-Archimedean-gog1}.
\end{proof}

\begin{lemma}
  The axiom \hyperref[gog2]{\textbf{GOG2}} holds in $\mathcal{G}_0$ and in $K^{>\mathbb{R}}$.
\end{lemma}

\begin{proof}
  Let $f, g \in K^{>\mathbb{R}}$ with $f, g \geqslant x$. Suppose first that
  $g \in \mathcal{G}_0$ and $f >\mathcal{C} (g)$. We may assume by
  \Cref{lem-G0-x+1} that $g = x + 1 + \varepsilon$ for an $\varepsilon \in
  K^{>} \cap \mathcal{o}$. We must have $f > x +\mathbb{R}$ by
  \Cref{lem-cof-centralizer}, whence $f \circ g > g \circ f$ by
  \Cref{lem-gog2-R}. Applying this for $f \in \mathcal{G}_0$, we see that
  \hyperref[gog2]{\textbf{GOG2}} holds in $\mathcal{G}_0$.
  
  Suppose now that $f >\mathcal{C} (g)$. If $g \in \mathcal{G}_0$, then the
  arguments above apply and yield $f \circ g > g \circ f$. If not, we have
  $\pi_1 (g) > x$ since $\mathcal{G}_0$ is a convex subgroup of
  $K^{>\mathbb{R}}$. Recall that $\mathcal{C} (\pi_1 (g)) \cap \mathcal{G}_1$
  is Archimedean, so $\pi_1 (g)^{[\mathbb{N}]}$ is cofinal in it. We have
  $\pi_1 (f) > \pi_1 (\mathcal{C}(g)) \supseteq \pi_1 (g^{[\mathbb{N}]}) =
  \pi_1 (g)^{[\mathbb{N}]}$, so $\pi_1 (f) >\mathcal{C} (\pi_1 (g)) \cap
  \mathcal{G}_1$. Thus \hyperref[gog2]{\textbf{GOG2}} in $\mathcal{G}_1$ yields
  \[ \pi_1 (f \circ g) = \pi_1 (f) \circ \pi_1 (g) > \pi_1 (g) \circ \pi_1 (f)
     = \pi_1 (g \circ f) . \]
  By \eqref{eq-lexico-quotient} and by convexity of $\mathcal{G}_0$, we
  have $f \circ g > g \circ f$. So \hyperref[gog2]{\textbf{GOG2}} holds.
\end{proof}

\begin{lemma}
   The axiom \hyperref[gog3]{\textbf{GOG3}} holds in $\mathcal{G}_0$ and in $K^{>\mathbb{R}}$.
\end{lemma}

\begin{proof}
  Let $g \in K^{>\mathbb{R}}$ with
  $g > x$. Suppose first that $g \in \mathcal{G}_0$. Let $\varphi \in
  \mathcal{G}_1$ with $\varphi^{\tmop{inv}} \circ g \circ \varphi = x + 1 +
  \varepsilon$ for some $\varepsilon \in K^{>} \cap \mathcal{o}$. By \Cref{prop-reals} and \Cref{lem-cof-centralizer}, the
 the element $x + 1$ is scaling in
  $K^{>\mathbb{R}}$ with $x + 1 \asymp \varphi^{\tmop{inv}} \circ g \circ
  \varphi$ in $K^{>\mathbb{R}}$. The conjugation by $\varphi$ is an
  automorphism of $(\mathcal{G}_0, \circ, x, <)$, so the element $\mathcal{s}
  \assign \varphi \circ (\varphi^{\tmop{inv}} + 1) \in \mathcal{G}_0$ is
  scaling in $\mathcal{G}_0$ with $\mathcal{s} \asymp g$. Thus $\mathcal{G}_0$ has scaling elements. In view of \Cref{lem-centraliser-nospill},
  we also obtain that $\mathcal{s}$ is also scaling in $K^{>\mathbb{R}}$ with
  $\mathcal{s} \asymp g$ in $K^{>\mathbb{R}}$.
  
  Now suppose that $g \nin \mathcal{G}_0$, so $g >\mathcal{G}_0$ by
  convexity. \ Let $\mathcal{t} \in \mathcal{G}_1$ be scaling in
  $\mathcal{G}_1$ with $\mathcal{t} \asymp \pi_1 (g)$ in $\mathcal{G}_1$.
  Since $\mathcal{C} (\pi_1 (g)) \cap \mathcal{G}_1$ is Archimedean, we have
  $\pi_1 (g)^{[- n]} \leqslant \mathcal{t} \leqslant \pi_1 (g)^{[n]}$ for some
  $n \in \mathbb{N}$, so $g^{[- n - 1]} \leqslant \mathcal{t} \leqslant g^{[n
  + 1]}$, whence $\mathcal{t} \asymp g$ in $K^{>\mathbb{R}}$. We claim that
  $\mathcal{t}$ is scaling in~$K^{>\mathbb{R}}$. Indeed let $f \in
  K^{>\mathbb{R}}$ with $f \asymp g$. By
  \Cref{prop-Archimedean-centralisers-here}, we have $g^{[- n]} \leqslant f
  \leqslant g^{[n]}$ for some $n \in \mathbb{N}$, so $\pi_1 (g^{[- n]})
  \leqslant \pi_1 (f) \leqslant \pi_1 (g^{[n]})$, whence $\pi_1 (f) \asymp
  \mathcal{t}$ in $\mathcal{G}_1$. Let $\mathcal{u} \in \mathcal{C}
  (\mathcal{t}) \cap \mathcal{G}_1$ with $\pi_1 (f) \mathcal{u}^{- 1} \prec
  \pi_1 (f)$ in $\mathcal{G}_1$. Since $\mathcal{G}_1$ has Archimedean
  centralisers, this means that $\pi_1 ((f\mathcal{u}^{- 1})^{[\mathbb{Z}]})=(\pi_1 (f)\mathcal{u}^{- 1})^{[\mathbb{Z}]} <
  \pi_1 (f)$,
  whence $(f\mathcal{u}^{- 1})^{[\mathbb{Z}]} < f$. We deduce with
  \Cref{prop-Archimedean-centralisers-here} that $f \sim \mathcal{u}$ in
  $K^{>\mathbb{R}}$. Thus $\mathcal{t}$ is scaling in
  $K^{>\mathbb{R}}$. Therefore \hyperref[gog3]{\textbf{GOG3}} holds in $K^{>\mathbb{R}}$.
\end{proof}

This concludes the proof of \Cref{th-levelled}.

\subsection{Application in the polynomially bounded
case}\label{subsection-application-bounded}

Let $\mathcal{R}$ be an o-minimal
expansion of $(\mathbb{R},+,\cdot,<)$. We
recall a fundamental dichotomy for the asymptotic growth of germs in
$\mathcal{G}_{\mathcal{R}}$:

\begin{descriptioncompact}
  \item[Miller's dichotomy] {\cite{Mil:dicho}} If there is an $f \in
  \mathcal{R}_{\infty}$ with $f > \tmop{id}^n$ for all $n \in \mathbb{N}$, then the
  exponential function is definable in $\mathcal{R}$.
\end{descriptioncompact}

If $\exp$ is not definable, then $\mathcal{R}$ is said
{\tmem{polynomially bounded}}. Let us first work on that smaller side of the dichotomy, that is,
suppose that $\mathcal{R}$ is polynomially bounded. Let $E$ denote the set of
real numbers $e$ such that the germ $\tmop{id}^e$ of the $e$-power function is
in $\mathcal{R}_{\infty}$. It is easy to see that $E$ is a subfield of
$\mathbb{R}$.

By {\cite[Proposition]{Mil:dicho}}, for each $f \in \mathcal{R}_{\infty}$,
there is a unique $(e_f, c_f) \in E \times \mathbb{R}$ such that $f - c_f
\tmop{id}^{e_f} \in \mathcal{o} (f)$. If $f >\mathbb{R}$, then we must have
$e_f > 0$ and $c_f > 0$. Note that $\mathbb{R}^{>}$ is an ordered
vector space over $E$, and thus we have a growth order group $\tmop{Aff}_E
(\mathbb{R}^{>})$ as in \Cref{ex-affine}. We set $\mathcal{G}_1 \assign \mathbb{R}^{>}
\tmop{id}^{E^{>}}$. Note that the function
\begin{eqnarray*}
  (e_{\cdot}, c_{\cdot}) \of \mathcal{G}_{\mathcal{R}} & \longrightarrow &
  \tmop{Aff}_E (\mathbb{R}^{>})\\
  f & \longmapsto & (e_f, c_f)
\end{eqnarray*}
is a homomorphism of ordered groups which restricts to an isomorphism
$\mathcal{G}_1 \longrightarrow \tmop{Aff}_E (\mathbb{R}^{>})$. Therefore $\mathcal{G}_1
\simeq \tmop{Aff}_E (\mathbb{R}^{>})$ is a growth order group. Let $\mathcal{G}_0$
denote the kernel of $(e_{\cdot}, c_{\cdot})$. So $\mathcal{G}_0$ is a normal
subgroup of $\mathcal{G}_{\mathcal{R}}$ and $\mathcal{G}_1$ is a complement of
$\mathcal{G}_0$ in $\mathcal{G}_{\mathcal{R}}$. Here $\mathcal{G}_0$ corresponds to germs that are tangent to the identity, whereas $\mathcal{G}_1$ is a group of non-monic monomials.

\begin{proposition}
  The ordered pair $(\mathcal{G}_0, \mathcal{G}_1)$ satisfies (\hyperref[dec]{$\mathbf{\star}$}) for
  $\mathcal{G}_{\mathcal{R}}$.
\end{proposition}

\begin{proof}
 We have $\mathcal{G}_0 = \{g \in
  \mathcal{G}_{\mathcal{R}} \suchthat g - x \in \mathcal{o}(\tmop{id})\}$, so $\mathcal{G}_0$ is a convex subgroup of $\mathcal{G}_{\mathcal{R}}$
  which contains $\tmop{id} + 1$. For $c \in \mathbb{R}^{>}
  \setminus \{1\}$, the centraliser of $(1, c)$ in $\tmop{Aff}_E (\mathbb{R}^{>})$ is $\{1\} \times \mathbb{R}^{>} \simeq \mathbb{R}^{>}$. In \Cref{ex-affine},
  we saw that given $e \in E^{>}$ with $e \neq 1$ and $c \in \mathbb{R}^{>}$, for all
  $q \in E^{>}$, there is a unique $c_0 \in \mathbb{R}^{>}$ such that $(q, c_0)$ and
  $(e, c)$ commute. Thus the projection on the first variable is an
  isomorphism between $\mathcal{C}((e, c))$ and
  $E^{>}$. Note that $E^{>}$ embeds into the Archimedean ordered
  group $(\mathbb{R}^{>}, \cdot, 1, <) \simeq (\mathbb{R}, +, 0, <)$, it is Archimedean. Therefore $\mathcal{G}_1$ has Archimedean centralisers.
  
  It remains to show that $L \assign \{f \circ (f^{\tmop{inv}} + 1)
  \suchthat f \in \mathcal{G}_1 \}$ is cofinal in $\mathcal{G}_0$. Let $g = x
  + \delta \in \mathcal{G}_0$, so $\delta \in \mathcal{o} (\tmop{id})$. We have $\delta
  - cx^e \in \mathcal{o} (\delta)$ for a certain $(e, c) \in E \times \mathbb{R}$. The condition
  $\delta \in \mathcal{o} (\tmop{id})$ implies that $e < 1$, so we find an $n \in
  \mathbb{N}$ with $\frac{2^n - 1}{2^n} > e$. Note that
  \[ \tmop{id}^{2^n} \circ (\tmop{id}^{2^{- n}} + 1) \in \tmop{id} + 2^n
     \tmop{id}^{\frac{2^n - 1}{2^n}} +\mathcal{o} (\tmop{id}^{\frac{2^n -
     1}{2^n}}) . \]
  Therefore $\tmop{id}^{2^n} \circ (\tmop{id}^{2^{- n}} + 1) > g$. This
  implies that $L$ is cofinal in $\mathcal{G}_0$.
\end{proof}

As $\mathcal{R}$ is polynomially bounded, \Cref{cor-Hardy-H} applies and entails that
$\mathcal{R}_{\infty}$ is an H-field with composition and inversion. 
\Cref{th-levelled} gives:

\begin{corollary}
  \label{cor-polynomially-bounded-case}Let $\mathcal{R}$ be a polynomially
  bounded o-minimal expansion of the real ordered field. Then
  $\mathcal{G}_{\mathcal{R}}$ is a growth order group with Archimedean
  centralisers.
\end{corollary}

\subsection{Applications in the exponential
case}\label{subsection-application-exp}

In order to deal with the exponential case, we introduce a notion of H-field
with an exponential function. We will also give additional applications of
\Cref{th-levelled}.

\begin{definition}
  An {\tmstrong{{\tmem{exponential H-field}}{\index{exponential H-field}}}} is
  an H-field $K$ over $\mathbb{R}$
  together with an isomorphism $\log \of (K^{>}, \cdot, 1, <) \longrightarrow
  (K, +, 0, <)$, whose reciprocal is denoted $\exp$, such that
  \begin{eqnarray}
    \log (1 +\mathcal{o}) & = & \mathcal{o} \text{\quad and} 
    \label{eq-compatible-exp}\\
    \forall a \in K^{>}, a^{\dag} & = & (\log a)' .  \label{eq-log-der}
  \end{eqnarray}
\end{definition}

Thus $(K, +, \cdot, 0, 1, <, \exp)$ is an ordered exponential field as per
{\cite{Kuhl:ordexp}}. We fix an exponential H-field $K$. Consider a Hardy
field with composition $\mathcal{H}$ containing $\log$ and a morphism of
ordered valued differential fields $\Phi \of \mathcal{H} \longrightarrow K$.
For all $f \in \mathcal{H}^{>}$, we have $\Phi (f) > 0$ and
\[ (\log \Phi (f))' = \frac{\Phi (f)'}{\Phi (f)} = \frac{\Phi (f')}{\Phi (f)}
   = \Phi (\frac{f'}{f}) = \Phi ((\log \circ f)') = \Phi (\log \circ f)' \]
by \eqref{eq-log-der}. So $\log \Phi (f) - \Phi (\log \circ f) \in \mathbb{R}$.
For all $a \in K^{>\mathbb{R}}$ and $\delta \in \mathcal{o} (a)$, we have $\log (a +
  \delta) - \log (s) \in \mathcal{o}$.
Indeed $\log (a + \delta) = \log (a (1 + \delta a^{- 1})) = \log (a)
+ \log (1 + \delta a^{- 1})$ where $\log (1 + \delta a^{- 1}) \in \mathcal{o}$
by \eqref{eq-compatible-exp}. An induction gives
\begin{equation}
  \log^{[k]} \Phi (f) - \Phi (\log^{[k]} \circ f) \in \mathcal{o}
  \label{eq-itelog}
\end{equation}
for all $f \in \mathcal{H}^{>\mathbb{R}}$ and $k > 1$.

\begin{proposition}
  \label{prop-main}Let $\mathcal{H}$ be a Hardy field with composition and
  inversion containing $\exp$ and let $\Phi \of \mathcal{H} \longrightarrow K$
  be embedding of ordered valued differential fields. Set $x \assign \Phi (\tmop{id})$ and suppose that for all $a
  \in K^{>\mathbb{R}}$, there is an $l \in \mathbb{Z}$ such that for all
  sufficiently large $k \in \mathbb{N}$, we have
  \begin{equation}
    \log^{[k]} (a) - \log^{[k-l]} (x) \in \mathcal{o}.
    \label{eq-level-cond}
  \end{equation}
  Then $\mathcal{H}^{>\mathbb{R}}$ is a growth order group with Archimedean
  centralisers.
\end{proposition}

\begin{proof}
  We will write $\mathcal{o}_K \assign \mathcal{o} (1) \subseteq K$ and
  $\mathcal{o}_{\mathcal{H}} \assign \mathcal{o} (1) \subseteq \mathcal{H}$.
  Consider the subgroup $\mathcal{G}_1
  \assign \exp^{[\mathbb{Z}]}$ of $\mathcal{H}^{>\mathbb{R}}$. This is a
  growth order group with Archimedean centralisers as it is itself Archimedean. Let $\mathcal{G}_0$ denote the subset of
  $\mathcal{H}^{>\mathbb{R}}$ of elements $g$ with $g^{[\mathbb{Z}]} < \exp$.
  This is a convex subgroup of $\mathcal{H}^{>\mathbb{R}}$ containing
  $\tmop{id} +\mathbb{R}$. We claim that $(\mathcal{G}_0, \mathcal{G}_1)$
  satisfies \hyperref[dec]{$(\mathbf{\star})$}. We have $\mathcal{G}_1 \cap \mathcal{G}_0 =\{\tmop{id}\}$ by definition. Let us show that $\mathcal{H}^{>\mathbb{R}}=\mathcal{G}_0 \mathcal{G}_1$. 
  
  Let $f \in \mathcal{H}$ with~$f \geqslant \tmop{id}$. By \eqref{eq-itelog}
  and \eqref{eq-level-cond}, we find an $l \in \mathbb{Z}$ such that for large
  enough $k > 1$, the element $\log^{[k]} (\Phi (f)) - \log^{[k]} (\exp^{[l]}
  (x))$ lies in $\mathcal{o}_K$. We claim that $g \assign f \circ \log^{[
  l]} \in \mathcal{G}_0$. By \eqref{eq-itelog}, given $k > 1$
  large enough, we have $\Phi (\log^{[k]} \circ f) - \Phi (\log^{[k-l]}
  (\tmop{id})) \in \mathcal{o}_K$, whence $\log^{[k]} \circ f - \log^{[k - l]}
  (\tmop{id}) \in \mathcal{o}_{\mathcal{H}}$. Thus $\log^{[k]}
  \circ f \circ \exp^{[k]} - \exp^{[l]}$ and
  $\log^{[k]} \circ g \circ \exp^{[k]} - \tmop{id}$ lie in $\mathcal{o}_{\mathcal{H}}$. 
 But then $\log^{[k]} \circ g \circ \exp^{[k]} \leqslant \tmop{id} + 1$ so
  $g^{[n]} < \exp^{[k]} \circ (\tmop{id} + n) \circ \log^{[k]} \leqslant
  \exp^{[k]} \circ \exp \circ \log^{[k]} = \exp$ for all $n \in \mathbb{N}$,
  i.e. $g \in \mathcal{G}_0$. 
  
  For $h \in \mathcal{G}_1$, $g \in \mathcal{G}_0$ and $n \in \mathbb{N}$, we have $(h \circ g
  \circ h^{\tmop{inv}})^{[n]} = h \circ g^{[n]} \circ
  h^{\tmop{inv}} < h \circ \exp \circ h^{\tmop{inv}} = \exp$. So $h \circ
  \mathcal{G}_0 \circ h^{\tmop{inv}} \subseteq \mathcal{G}_0$. It follows since $\mathcal{H}^{>\mathbb{R}} =\mathcal{G}_0
  \mathcal{G}_1$ that $\mathcal{G}_0$ is a normal subgroup of
  $\mathcal{H}^{>\mathbb{R}}$.
  
  Finally, assume for contradiction that
  $g > \exp^{[k]} \circ (\log^{[k]} + 1)$ for some $g \in \mathcal{G}_0$, for all $k > 1$. By
  \eqref{eq-itelog}, for each $k > 1$, we have a $\delta_k \in \mathcal{o}_K$
  with $\log^{[k]} (\Phi (g)) + \delta_k > \log^{[k]} (x) + 1$. In particular
  $\log^{[k]} \Phi (g) > \log^{[k]} (x) + \frac{1}{2}$, whence
  $\Phi (g) > \exp^{[k]} (\log^{[k]} (x) + \frac{1}{2})$, for all $k > 1$. Let $\ell \in \mathbb{Z}$ and $k_0 > 1$ with $\log^{[k_0]} (\Phi (g)) -
  \log^{[k_0-\ell]} (x) \in \mathcal{o}_K$. We have $\ell>0$ since $\log^{[k_0]} (x) + \frac{1}{4} < \log^{[k_0-\ell]} (x)$. Now \eqref{eq-itelog} gives $\Phi
  (\log^{[k_0]} (g) - \log^{[k_0-\ell]}) \in \mathcal{o}_K$, so $\log^{[k_0]}
  (g) - \log^{[k_0-\ell]} \in \mathcal{o}_{\mathcal{H}}$. In particular
  $\log^{[k_0]} (g) - \log^{[k_0-\ell]} \geqslant - 1$, thus
  \begin{eqnarray*}
    g^{[2]} & = & \exp^{[k_0]} (\log^{[k_0]} (g)) \circ \exp^{[k_0]}
    (\log^{[k_0]} (g))\\
    & \geqslant & \exp^{[k_0]} \circ (\tmop{id} - 1) \circ \log^{[k_0-\ell]} \circ \exp^{[k_0]} \circ (\tmop{id} - 1) \circ \log^{[k_0-\ell]}\\
    & \geqslant & \exp^{[k_0]} \circ ((\tmop{id} - 1) \circ \exp^{[\ell]}
    \circ (\tmop{id} - 1) \circ \exp^{[\ell]}) \circ \log^{[ k_0]} .
  \end{eqnarray*}
  We have $(\tmop{id} - 1) \circ \exp^{[\ell]} \circ (\tmop{id} - 1) \circ \exp^{[\ell]})= h \circ \exp^{[2 \ell]}$,
  where $h \assign (\tmop{id} - 1) \circ (\exp^{[\ell]} \circ (\tmop{id} - 1)
  \circ \log^{[ \ell]})$. Now $h \in \mathcal{G}_0$ by our previous
  arguments, so $h \geqslant \log$, so $g^{[2]} \geqslant \exp^{[k_0]} \circ h
  \circ \exp^{[2 \ell]} \circ \log^{[ k_0]} \geqslant \exp^{[2\ell-1]}$. This contradicts the assumption that $g \in \mathcal{G}_0$, and thus
  concludes out proof that $\{f \circ (f^{\tmop{inv}} + 1) \suchthat f \in
  \mathcal{G}_1 \}$ is cofinal in $\mathcal{G}_0$. So \hyperref[dec]{$(\mathbf{\star})$} holds. We conclude with \Cref{th-levelled}.
\end{proof}

\begin{corollary}
  \label{cor-levelled-Pfaff}Let $\mathcal{P}$ be the Pfaffian closure of
  the real ordered field {\tmem{{\cite{Spei:Pfaff}}}}. Then
  $\mathcal{G}_{\mathcal{P}}$ is a growth order group with Archimedean
  centralisers.
\end{corollary}

\begin{proof}
  The field $\mathbb{T}_{\tmop{LE}}$ of logarithmic-exponential transseries is
  an exponential H-field (see {\cite{vdDMM01,vdH:mt}}). The property \eqref{eq-level-cond} holds {\cite[Claim, p
  248]{MarMil97}} in
  $\mathbb{T}_{\tmop{LE}}$. We have an embedding of ordered valued
  differential fields {\cite[Corollary~7.3.4]{ADH:H-closed}} of
  $\mathcal{P}_{\infty}$ into $\mathbb{T}_{\tmop{LE}}$. So \Cref{prop-main}
  applies.
\end{proof}

Let us complete our proof of \Cref{th-main}.
Let $\mathcal{R}$ be a levelled expansion of the real ordered field that is not polynomially bounded. We have $\exp,\log \in
\mathcal{G}_{\mathcal{R}}$ by Miller's dichotomy. This yields an isomorphism
of ordered groups
\begin{eqnarray*}
  \log \of \mathcal{R}_{\infty}^{>} & \longrightarrow & \mathcal{R}_{\infty}\\
  f & \longmapsto & \log \circ f
\end{eqnarray*}
and $(\mathcal{R}_{\infty}, \log)$ is an exponential H-field (see \cite[Section 6.2]{Kuhl:ordexp}). Since $\mathcal{R}$ is levelled and in view of \eqref{eq-itelog}, the condition \eqref{eq-level-cond} holds. \Cref{prop-main} gives:

\begin{corollary}
  \label{cor-exponential-case}Let $\mathcal{R}$ be a levelled o-minimal
  expansion of the real ordered field that is not polynomially bounded. Then
  $\mathcal{G}_{\mathcal{R}}$ is a growth order group with Archimedean
  centralisers.
\end{corollary}

\begin{remark}
    Any reduct of a levelled o-minimal expansion of the real ordered field that defines the sum and product is clearly a levelled o-minimal expansion of the real ordered field, therefore it also induces a growth order group.
\end{remark}

 Corollaries~\ref{cor-polynomially-bounded-case}
and~\ref{cor-exponential-case} imply \Cref{th-main}. By {\cite[Theorem~1]{KK:valH1}, we have:

\begin{corollary}
  \label{cor-levelled-qaa}Let $\mathcal{R}$ be an o-minimal expansion of the
  real ordered field by the exponential and a generalised quasianalytic class
  {\tmem{{\cite{RoSer:QEqaa}}}} containing the restricted analytic $\exp$ and $\log$. Then $\mathcal{G}_{\mathcal{R}}$ is a growth order group with
  Archimedean centralisers.
\end{corollary}

\begin{acknowledgments*} We thank Lou van den Dries, Fran{\c c}oise Point and Tamara Servi for their
  answers to our questions. We thank Sylvy Anscombe for her precious advice.
\end{acknowledgments*}


\begin{thebibliography}{10}
  \bibitem[1]{AvdD02}M.~Aschenbrenner  and  L.~van~den Dries.
  {\newblock}H-fields and their Liouville extensions.
  {\newblock}\textit{Mathematische Zeitschrift}, 242(3):543--588,
  2002, \href{https://doi.org/10.1007/s002090000358}{https://doi.org/10.1007/s002090000358}.
  
  \bibitem[2]{AvdD03}M.~Aschenbrenner  and  L.~van~den Dries.
  Liouville closed H-fields. \textit{Journal of Pure
  and Applied Algebra}, 197:1--55, 2003, \href{https://doi.org/10.1016/j.jalgebra.2023.03.019}{https://doi.org/10.1016/j.jalgebra.2023.03.019}.
  
  \bibitem[3]{vdH:mt}M.~Aschenbrenner, L.~van~den Dries, and  J.~van~der
  Hoeven. \textit{Asymptotic Differential Algebra and Model
  Theory of Transseries}. Number  195 in Annals of Mathematics
  studies. Princeton University Press, 2017, \href{https://doi.org/10.23943/princeton/9780691175423.001.0001}{https://doi.org/10.23943/princeton/9780691175423.001.0001}{\newblock}
  
  \bibitem[4]{ADH:H-closed}M.~Aschenbrenner, L.~van~den Dries, and  J.~van~der
  Hoeven. {\newblock}Maximal Hardy fields.
  {\newblock}\href{https://arxiv.org/abs/2304.10846}{\href{Https://arxiv.org/abs/2304.10846}{https://arxiv.org/abs/2304.10846}},
  2023.{\newblock}
  
  \bibitem[5]{Bag:hyperclosed}V.~Bagayoko. {\newblock}Hyperexponentially
  closed fields.
  {\newblock}\href{https://www.hal.inserm.fr/X-LIX/hal-03686767v1}{\href{https://www.hal.inserm.fr/X-LIX/hal-03686767v1}{https://www.hal.inserm.fr/X-LIX/hal-03686767v1}},
  2022.{\newblock}
  
  \bibitem[6]{Bag:phd}V.~Bagayoko. {\newblock}\tmtextit{Hyperseries and
  surreal numbers}. {\newblock}PhD thesis, UMons, Ecole Polytechnique, 2022.
  {\newblock}\href{https://theses.hal.science/tel-04105359}{\href{Https://theses.hal.science/tel-04105359/}{https://theses.hal.science/tel-04105359}}.{\newblock}

  \bibitem[7]{Bag:mg}V.~Bagayoko. Groups with infinite linearly ordered products.
  \href{https://arxiv.org/abs/2403.07368}{\href{https://arxiv.org/abs/2403.07368}{https://arxiv.org/abs/2403.07368}},
  2024.{\newblock}
  
  \bibitem[8]{BaiBalVer}B.~Baizhanov, J.~Baldwin, and  V.~Verbovskiy.
  Cayley's theorem for ordered groups: o-minimality.\textit{Sibirskie {\`E}lektronnye Matematicheskie Izvestiya
  [electronic only]}, 4:278--281, 2007.
  
  \bibitem[9]{Bosh82}M.~Boshernitzan. New ``orders of infinity''.
  \textit{Journal d'Analyse Math{\'e}matique}, 41:130--167,
  1982, \href{https://doi.org/10.1007/bf02803397}{https://doi.org/10.1007/bf02803397}.
  
  \bibitem[10]{Bou07}N.~Bourbaki. \textit{Fonctions d'une
  variable r{\'e}elle: Th{\'e}orie {\'e}l{\'e}mentaire}.
  El{\'e}ments de math{\'e}matique. Springer Berlin Heidelberg,
  2007.
  
  \bibitem[11]{Con76}J.~H.~Conway. \textit{On numbers and games}.
  Academic Press, 1976, \href{https://doi.org/10.1201/9781439864159}{https://doi.org/10.1201/9781439864159}.
  
  \bibitem[12]{vdD98}L.~van den, Dries. \textit{Tame topology and
  o-minimal structures},  volume  248  of \textit{London Math. Soc. Lect.
  Note}. Cambridge University Press, 1998, \href{https://doi.org/10.1017/cbo9780511525919}{https://doi.org/10.1017/cbo9780511525919}.
  
  \bibitem[13]{vdDMM01}L.~van~den Dries, A.~Macintyre, and  D.~Marker.
  Logarithmic-exponential series. \textit{Annals of
  Pure and Applied Logic}, 111:61--113, 07 2001, \href{https://doi.org/10.1016/S0168-0072(01)00035-5}{https://doi.org/10.1016/S0168-0072(01)00035-5}.
  
  \bibitem[14]{Ec92}J.~{\'E}calle. \textit{Introduction aux
  fonctions analysables et preuve constructive de la conjecture de Dulac}.
  Actualit{\'e}s Math{\'e}matiques. Hermann, 1992.
  
  \bibitem[15]{FiRos:CT}B.~Fine  and  G.~Rosenberger. \textit{Reflections on Commutative Transitivity}, chapter 8, pages 112--130. World Scientific Press, 2008, \href{https://doi.org/10.1142/9789812793416_0008}{https://doi.org/10.1142/9789812793416-0008}.
  
  \bibitem[16]{Fuchs11}L.~Fuchs. \textit{Partially Ordered
  Algebraic Systems}. Dover Publications, 2011.
  
  \bibitem[17]{Glass}A.~M.~W.~Glass. \textit{Partially ordered
  groups},  volume~7  of \textit{Series in Algebra}. World
  Scientific Publisher Co Pte Ltd, 1999, \href{https://doi.org/10.1142/3811}{https://doi.org/10.1142/3811}.
  
  \bibitem[18]{Hahn1907}H.~Hahn. {\"U}ber die nichtarchimedischen
  Gr{\"o}$\beta$ensysteme. \textit{Sitz. Akad. Wiss. Wien},
  116:601--655, 1907, \href{https://doi.org/10.1007/978-3-7091-6601-7_18}{https://doi.org/10.1007/978-3-7091-6601-7-18}.
  
  \bibitem[19]{H1910}G.~H.~Hardy. \textit{Orders of infinity, the
  'Infinit{\"a}rcalc{\"u}l' of Paul du Bois-Reymond}. Cambridge
  University Press  edition, 1910.
  
  \bibitem[20]{vdH:phd}J.~van~der Hoeven. \textit{Automatic
  asymptotics}. PhD thesis, {\'E}cole polytechnique, Palaiseau,
  France, 1997.
  
  \bibitem[21]{vdH:ln}J.~van~der Hoeven. \textit{Transseries and
  real differential algebra},  volume  1888  of \textit{Lecture Notes in
  Mathematics}. Springer-Verlag, 2006, \href{https://doi.org/10.1007/3-540-35590-1}{https://doi.org/10.1007/3-540-35590-1}.
  
  \bibitem[22]{Iwasawa:free}K.~Iwasawa. On linearly ordered groups.
  \textit{J. Math. Soc. Japan}, 1(1):1--9, 1948, \href{https://doi.org/10.2969/jmsj/00110001}{https://doi.org/10.2969/jmsj/00110001}.
  
  \bibitem[23]{KK:valH1}F.--V.~Kuhlmann  and  S.~Kuhlmann.
  Valuation theory of exponential Hardy fields I. \textit{Mathematische Zeitschrift}, 243:671--688,
  2003, \href{https://doi.org/10.1007/s00209-002-0460-4}{https://doi.org/10.1007/s00209-002-0460-4}.
  
  \bibitem[24]{Kuhl:ordexp}S.~Kuhlmann. \textit{Ordered
  exponential fields},  volume~12  of \textit{Field Institute Monographs}.
  American Mathematical Society, 2000, \href{https://doi.org/10.1090/fim/012}{https://doi.org/10.1090/fim/012}.
  
  \bibitem[25]{Levi:ord-grp}F.~W.~Levi. Ordered groups. \textit{Proceedings of the Indian Academy of Sciences - Section
  A}, 16(4):256--263, 1942, \href{https://doi.org/10.1007/bf03174799}{https://doi.org/10.1007/bf03174799}.
  
  \bibitem[26]{Loonstra:ord-grp}F.~Loonstra. Ordered groups.
  \textit{Proc. Nederl. Akad. Wetensch.}, 49:41--46,
  1946.
  
  \bibitem[27]{MarMil97}D.~Marker  and  C.~Miller. Levelled
  o-minimal structures. \textit{Revista Matematica de la
  Universidad Complutense de Madrid}, 10, 1997, \href{https://doi.org/10.5209/rev_rema.1997.v10.17371}{https://doi.org/10.5209/rev-rema.1997.v10.17371}.
  
  \bibitem[28]{Mil:dicho}C.~Miller. Exponentiation is hard to
  avoid. \textit{Proc. of the Am. Math. Soc.}, 122(1):257--259,
  1994, \href{https://doi.org/10.2307/2160869}{https://doi.org/10.2307/2160869}.

  \bibitem[29]{MilStar:growth} C.~Miller and S.~Starchenko. A Growth Dichotomy for O-Minimal Expansions of Ordered Groups. \textit{Trans. of the Am. Math. Soc.}, 350(9), 3505--3521, 1998, \href{https://doi.org/10.1090/s0002-9947-98-02288-0}{https://doi.org/10.1090/s0002-9947-98-02288-0}.
  
  \bibitem[30]{MuRhem}R.~B.~Mura  and  A.~Rhemtulla. \textit{Orderable groups}. Lecture Notes in Pure and
  Applied Mathematics. \begin{tabular}{c}
    l Taylor Francis
  \end{tabular}, 1977.
  
  \bibitem[31]{Neu:or-gr}B.~H.~Neumann. On ordered groups. \textit{American Journal of Mathematics}, 71(1):1--18,
  1949, \href{https://doi.org/10.2307/2372087}{https://doi.org/10.2307/2372087}.

  \bibitem[32]{Pad:phd}A.~Padgett. \textit{Sublogarithmic-transexponential series}. PhD
  thesis, Berkeley, 2022.
  
  \bibitem[33]{PilStein:omin}A.~Pillay  and  C.~Steinhorn.
  Definable Sets in Ordered Structures. I. \textit{Transactions of the American Mathematical Society},
  295(2):565--592, 1986, \href{https://doi.org/10.2307/2000052}{https://doi.org/10.2307/2000052}.
  
  \bibitem[34]{Rob:complete}A.~Robinson. \textit{Complete
  theories}. North Holland, 1956.
  
  \bibitem[35]{RoSer:QEqaa}J.-P.~Rolin  and  T.~Servi. Quantifier
  elimination and rectilinearization theorem for generalized quasianalytic
  algebras. \textit{Proceedings of the London Mathematical
  Society}, 110(5):1207--1247, 2015, \href{https://doi.org/10.1112/plms/pdv010}{https://doi.org/10.1112/plms/pdv010}.
  
  \bibitem[36]{RoSerSpei22}J.-P.~Rolin, T.~Servi, and  P.~Speissegger.
  Multisummability for generalized power series.
  \textit{Canadian Journal of Mathematics}, (1):1--37,
  2023, \href{https://doi.org/10.4153/s0008414x23000111}{https://doi.org/10.4153/s0008414x23000111}.

  \bibitem[37]{Rosli80}M.~Rosenlicht. Differential valuations.
  \textit{Pacific Journal of Mathematics},
  86:301--319, 1980, \href{https://doi.org/10.2140/pjm.1980.86.301}{https://doi.org/10.2140/pjm.1980.86.301}.
  
  \bibitem[38]{Rosli83}M.~Rosenlicht. Hardy fields. \textit{Journal of Mathematical Analysis and Applications},
  93(2):297--311, 1983, \href{https://doi.org/10.1016/0022-247X(83)90175-0}{https://doi.org/10.1016/0022-247X(83)90175-0}.

  \bibitem[39]{Rosli83:rank}M.~Rosenlicht. The rank of a Hardy field. \textit{Trans. of the American Math. Soc.},
  280(2):659--671, 1983, \href{https://doi.org/10.1090/s0002-9947-1983-0716843-5}{https://doi.org/10.1090/s0002-9947-1983-0716843-5}.

  \bibitem[40]{SaSo:vgroups}M.~Saarimäki and P. Sorjonen {\newblock}Valued groups. {\newblock}\tmtextit{Mathematica Scandinavica}, 70(2):265--280, 1992, \href{https://doi.org/10.7146/math.scand.a-12401}{https://doi.org/10.7146/math.scand.a-12401}.
  
  \bibitem[41]{Spei:Pfaff}P.~Speissegger. The Pfaffian closure of
  an o-minimal structure. \textit{Journal f{\"u}r die reine und
  angewante Mathematik}, 1999:189--211, 1997, \href{https://doi.org/10.1515/crll.1999.508.189}{https://doi.org/10.1515/crll.1999.508.189}.

  \bibitem[42]{Tyne:phd}J.~Tyne. \textit{T-levels and T-convexity.}
  PhD thesis, UIUC, 2003.

  \bibitem[43]{Tyne:height}J.~Tyne. T-height in weakly o-minimal structures.
  \textit{The Journal of Symbolic Logic},
  71(3):747--762, 2006, \href{https://doi.org/10.2178/jsl/1154698574}{https://doi.org/10.2178/jsl/1154698574}.
  
\end{thebibliography}
\end{document}